\documentclass[reqno]{amsart}
\usepackage{amsmath, amssymb, amsthm, epsfig}
\usepackage{hyperref}
\usepackage{latexsym}
\usepackage{url} 

\usepackage{color}

\DeclareMathOperator{\sgn}{\mathrm{sgn}}

\begin{document}
 \bibliographystyle{plain}

 \newtheorem{theorem}{Theorem}
 \newtheorem{lemma}[theorem]{Lemma}
 \newtheorem{proposition}[theorem]{Proposition}
 \newtheorem{corollary}[theorem]{Corollary}
 \theoremstyle{definition}
 \newtheorem{definition}[theorem]{Definition}
 \newtheorem{example}[theorem]{Example}
 \theoremstyle{remark}
 \newtheorem{remark}[theorem]{Remark}
 \newcommand{\mc}{\mathcal}
 \newcommand{\A}{\mc{A}}
 \newcommand{\B}{\mc{B}}
 \newcommand{\cc}{\mc{C}}
 \newcommand{\D}{\mc{D}}
 \newcommand{\E}{\mc{E}}
 \newcommand{\F}{\mc{F}}
 \newcommand{\G}{\mc{G}}
 \newcommand{\sH}{\mc{H}}
 \newcommand{\I}{\mc{I}}
 \newcommand{\J}{\mc{J}}
 \newcommand{\nn}{\mc{N}}
 \newcommand{\rr}{\mc{R}}
 \newcommand{\sS}{\mc{S}}
 \newcommand{\U}{\mc{U}}
 \newcommand{\X}{\mc{X}}
 \newcommand{\Y}{\mc{Y}}
 \newcommand{\C}{\mathbb{C}}
 \newcommand{\R}{\mathbb{R}}
 \newcommand{\N}{\mathbb{N}}
 \newcommand{\Q}{\mathbb{Q}}
 \newcommand{\Z}{\mathbb{Z}}
 \newcommand{\csch}{\mathrm{csch}}
 \newcommand{\tF}{\widehat{F}}
 \newcommand{\tG}{\widehat{G}}
 \newcommand{\tH}{\widehat{H}}
 \newcommand{\tf}{\widehat{f}}
 \newcommand{\ug}{\widehat{g}}
 \newcommand{\wg}{\widetilde{g}}
 \newcommand{\uh}{\widehat{h}}
 \newcommand{\wh}{\widetilde{h}}
 \newcommand{\wl}{\widetilde{l}}
 \newcommand{\tk}{\widehat{k}}
 \newcommand{\tK}{\widehat{K}}
 \newcommand{\tl}{\widehat{l}}
 \newcommand{\tL}{\widehat{L}}
 \newcommand{\tm}{\widehat{m}}
 \newcommand{\tM}{\widehat{M}}
 \newcommand{\tp}{\widehat{\varphi}}
 \newcommand{\tq}{\widehat{q}}
 \newcommand{\tT}{\widehat{T}}
 \newcommand{\tU}{\widehat{U}}
 \newcommand{\tu}{\widehat{u}}
 \newcommand{\tV}{\widehat{V}}
 \newcommand{\tv}{\widehat{v}}
 \newcommand{\tW}{\widehat{W}}
 \newcommand{\ba}{\boldsymbol{a}}
 \newcommand{\bal}{\boldsymbol{\alpha}}
 \newcommand{\bx}{\boldsymbol{x}}
 \newcommand{\p}{\varphi}
 \newcommand{\f}{\frac52}
 \newcommand{\g}{\frac32}
 \newcommand{\h}{\frac12}
 \newcommand{\hh}{\tfrac12}
 \newcommand{\ds}{\text{\rm d}s}
 \newcommand{\dt}{\text{\rm d}t}
 \newcommand{\du}{\text{\rm d}u}
 \newcommand{\dv}{\text{\rm d}v}
 \newcommand{\dw}{\text{\rm d}w}
  \newcommand{\dz}{\text{\rm d}z}
 \newcommand{\dx}{\text{\rm d}x}
 \newcommand{\dy}{\text{\rm d}y}
 \newcommand{\dl}{\text{\rm d}\lambda}
 \newcommand{\dmu}{\text{\rm d}\mu(\lambda)}
 \newcommand{\dnu}{\text{\rm d}\nu(\lambda)}
  \newcommand{\dnuN}{\text{\rm d}\nu_N(\lambda)}
\newcommand{\dnus}{\text{\rm d}\nu_{\sigma}(\lambda)}
 \newcommand{\dlnu}{\text{\rm d}\nu_l(\lambda)}
 \newcommand{\dnnu}{\text{\rm d}\nu_n(\lambda)}
\newcommand{\sech}{\text{\rm sech}}
\newcommand{\CC}{\mathbb{C}}
\newcommand{\NN}{\mathbb{N}}
\newcommand{\RR}{\mathbb{R}}
\newcommand{\ZZ}{\mathbb{Z}}
\newcommand{\thp}{\theta^+}
\newcommand{\thpn}{\theta^+_N}
\newcommand{\vthp}{\vartheta^+}
\newcommand{\vthpn}{\vartheta^+_N}
\newcommand{\ft}[1]{\widehat{#1}}
\newcommand{\support}[1]{\mathrm{supp}(#1)}
\newcommand{\gplus}{G^+_\lambda}
\newcommand{\ph}{\mathcal{H}}
\newcommand{\godd}{G^o_\lambda}
\newcommand{\mplus}{M_\lambda^+}
\newcommand{\lplus}{L_\lambda^+}
\newcommand{\modd}{M_\lambda^o}
\newcommand{\lodd}{L_\lambda^o}
\newcommand{\sgp}{x_+^0}
\newcommand{\Tl}{T_\lambda}
\newcommand{\Tlc}{T_{\lambda,c}}
\newcommand{\El}{{E_\lambda}}

 \newcommand{\T}{\mc{T}}
 \newcommand{\M}{\mc{M}}
 \renewcommand{\L}{\mc{L}}
 \newcommand{\K}{\mc{K}}
 \renewcommand{\H}{\mc{H}}

\newcommand{\TT}{\mathfrak{T}}
\newcommand{\MM}{\mathfrak{M}}
\newcommand{\LL}{\mathfrak{L}}
\newcommand{\KK}{\mathfrak{K}}
\newcommand{\GG}{\mathfrak{G}}
\newcommand{\HH}{\mathfrak{H}}

\newcommand{\newr}[1]{\textcolor{red}{#1}}
\newcommand{\new}[1]{\textcolor{black}{#1}}

 \def\today{\ifcase\month\or
  January\or February\or March\or April\or May\or June\or
  July\or August\or September\or October\or November\or December\fi
  \space\number\day, \number\year}

\title[Approximations to the truncated functions]{Entire approximations for a class of truncated and odd functions}
\author[Carneiro and Littmann]{Emanuel Carneiro and  Friedrich Littmann}

\date{\today}

\thanks{E. C. acknowledges support from CNPq-Brazil grants $473152/2011-8$ and $302809/2011-2$, and FAPERJ grant E-26/103.010/2012. Part of this work was completed during a visit of F.L. to IMPA, for which he gratefully acknowledges the support.}

\address{IMPA - Instituto de Matem\'{a}tica Pura e Aplicada, Estrada Dona Castorina, 110, Rio de Janeiro, Brazil 22460-320.}
\email{carneiro@impa.br}
\address{Department of mathematics, North Dakota State University, Fargo, ND 58105-5075.}
\email{friedrich.littmann@ndsu.edu}

\begin{abstract} We solve the problem of finding optimal entire approximations  of prescribed exponential type (unrestricted, majorant and minorant) for a class of truncated and odd functions with a shifted exponential subordination, minimizing the $L^1(\R)$-error. The class considered here includes new examples such as the truncated logarithm and truncated shifted power functions. This paper is the counterpart of the works \cite{CV2} and \cite{CV3} where the analogous problem for even functions was treated.

\end{abstract}

\maketitle

\numberwithin{equation}{section}

\section{Introduction}

We address in this paper the extremal problems of finding optimal entire functions of prescribed exponential type that approximate, majorize or minorize a given function $f:\R \to \R$, minimizing the $L^1(\R)$-error. Recall that an entire function $K:\C \to \C$ has exponential type $\sigma \geq 0$ if for every $\epsilon >0$ there exists $C_{\epsilon} >0$ such that 
\begin{equation*}
|K(z)| \leq C_{\epsilon} \,e^{(\sigma + \epsilon) |z|} 
\end{equation*}
for all $z \in \C$.

\smallskip

The best approximation problem (also referred to as two-sided approximation) is a classical problem in approximation theory and harmonic analysis, and dates back to the works of Krein \cite{K} and Sz.-Nagy \cite{Na}. The extremal majorant/minorant problem (also referred to as one-sided approximations) was independently introduced by Beurling for the function $\sgn(x)$ in connection with bounds for almost periodic functions (see \cite{V}).  With the observation that $\chi_{[a,b]}(x) = \tfrac{1}{2}\{\sgn(x-a) + \sgn(b-x)\}$, Selberg constructed majorants and minorants for characteristic functions of intervals, a simple yet very useful tool for number theoretic applications, cf.\ \cite{BMV, Ga, GG, S2, V}. For other works related to this theory of extremal functions of exponential type and its applications we refer to \cite{Gan, HZ, HV, L1, L3, L4} and the references therein.

\smallskip

The idea of extending the solution of the extremal problem from a base case with a free paramater to a whole class of functions via an integration argument is due to Graham and Vaaler \cite{GV}. In this work they solve the extremal problem for the exponential $e^{-\lambda |x|}$ (also for its truncated and odd versions), and are able to integrate the parameter $\lambda >0$ against nonnegative Borel measures $\dnu$ that satisfy certain integrability conditions. In the case of even functions this was later refined in \cite{CV2, CV3} when Carneiro and Vaaler considered a shifted exponential $e^{-\lambda |x|} - e^{-\lambda}$ to be able to integrate it against the optimal class of measures  $\dnu$ and obtain the solution for a class of even functions that includes $\log |x|$, for instance. The most general framework for solving this problem in the case of even functions was later obtained by Carneiro, Littmann and Vaaler in \cite{CLV} with the solution for the Gaussian $e^{-\lambda \pi x^2}$, and the extension to a larger class of even functions via integration and the use of distribution theory. The special family of even functions $\log\big[(x^2 + a^2)/(x^2 + b^2)\big]$, contemplated by this method, was later used to improve the known bounds for the Riemann zeta function on the critical strip under the Riemann hypothesis \cite{CC, CS}. 

\smallskip 

In the case of truncated and odd functions the picture is different. Recently, the authors \cite{CL} have solved the problem for the truncated and odd Gaussians, and extended the construction to a general class of truncated and odd functions. The special odd function $\arctan(1/x) - x/(1 + x^2)$, that falls under the scope of this framework, was later applied to obtain improved bounds for the argument of the Riemann zeta function on the critical line under the Riemann hypothesis \cite{CCM}. However, the Gaussian subordination is not as powerful in the truncated and odd setting as it is for the even setting (the reason for this is the very fast decay of the Fourier transform of the even Gaussian $e^{-\lambda \pi x^2}$ as $\lambda \to 0$), and examples like the truncated logarithm are not reached by this method.

\smallskip

The purpose of this paper is then to construct an analogous theory as in \cite{CV2, CV3} for the case of truncated and odd functions, exploring the exponential subordination. A major difficulty in accomplishing this task is dealing with the appropriate shift, which is not a constant anymore but a step function. The class of functions that we achieve does indeed contain previously inaccessible functions such as the truncated logarithm and truncated shifted power functions, and this should close the circle with the works \cite{CL, CLV, CV2, CV3} in producing the most general framework for even, odd and truncated functions with exponential or Gaussian subordination.

\section{Results}

\subsection{The base case} To describe the functions for which we solve these extremal problems we let $\lambda >0$, $c\in\R$ and consider  $\Tlc:\R \to \R$ given by
\begin{equation*}
\Tlc(x) = \left\{
\begin{array}{rl}
e^{-\lambda x} - c & \textrm{if} \ \ x>0,\\
\frac{1}{2}(1 - c) & \textrm{if} \ \ x=0,\\
0 & \textrm{if} \ \ x<0.
\end{array}
\right.
\end{equation*}
We define the entire functions $K_{\lambda,c}$, $L_{\lambda,c}$ and $M_{\lambda,c}$ of exponential type $2\pi$ by
\begin{align}\label{defKLM}
\begin{split}
K_{\lambda,c}(z)&=\frac{\sin\pi z}{\pi} \left\{ \sum_{n=1}^\infty (-1)^{n} \left(\frac{\Tlc(n)}{(z-n)} -\frac{e^{-\lambda n}}{z}\right) - \frac{c}{2z}\right\},\\
L_{\lambda,c}(z) &= \frac{\sin^2\pi z}{\pi^2}\left\{ \sum_{n=1}^\infty \left(\frac{\Tlc(n)}{(z-n)^2} + \frac{\Tlc'(n)}{(z-n)}-\frac{\Tlc'(n)}{z}\right) - \frac{c}{z}\right\},\\
M_{\lambda,c}(z) &= \frac{\sin^2\pi z}{\pi^2} \left\{\sum_{n=1}^\infty \left(\frac{\Tlc(n)}{(z-n)^2} + \frac{\Tlc'(n)}{(z-n)}-\frac{\Tlc'(n)}{z}\right) - \frac{c}{z}\  + \frac{1-c}{z^2} \right\}.
\end{split}
\end{align}

The motivation for this particular choice of $K_{\lambda,c}$, $L_{\lambda,c}$ and $M_{\lambda,c}$ is the fact that $T_{\lambda,c}$ itself is the difference of two truncated functions, an exponential and a step function, and for each of these separately we can solve the extremal problem \cite{GV, V}. In general, the difference of two majorants is of course not the majorant of the difference. However, it turns out that in the case of $\Tlc$, for a sufficiently broad range of parameters $\lambda$ and $c$, the difference is indeed the optimal solution. We take $\lambda$ to be the independent parameter and express $c$ as a function of $\lambda$ and of the exponential type.

\begin{theorem}[Optimal two-sided approximation]\label{tr-ba-thm} 
Let $\delta >0$, $\lambda >0$ and $c\le e^{-\delta^{-1}\lambda}$. The inequality
\begin{align}\label{th1-eq1}
\sin \pi\delta x \, \big\{\Tlc(x)-K_{\delta^{-1}\lambda,c}(\delta x)\big\} \ge 0
\end{align}
holds for all real $x$.  If $K$ is an entire function of type $\pi\delta$, then
\begin{align}\label{th1-eq2}
\int_{-\infty}^\infty \big|\Tlc(x)- K(x)\big| \, \dx \ge  \frac{1}{\delta} \left( \frac{1-e^{-\delta^{-1}\lambda}}{\delta^{-1}\lambda \big(1+e^{-\delta^{-1}\lambda}\big)} - \frac{c}{2}\right),
\end{align}
with equality if and only if $K(z) = K_{\delta^{-1}\lambda,c}(\delta z)$ for all $z \in \C$.
\end{theorem}

\smallskip

\begin{theorem}[Optimal one-sided approximations] \label{tr-majorant-thm}

Let $\delta >0$ and $\lambda >0$.

\smallskip

\noindent {\rm (i) (Minorant)}  Let $c\le e^{-\delta^{-1}\lambda}$. The inequality
\begin{align}\label{th2-eq1}
L_{\delta^{-1}\lambda,c}(\delta x) \le \Tlc(x) 
\end{align}
holds for all real $x$. If $L$ is an entire function of exponential type $2\pi\delta$ satisfying $L(x)\le \Tlc(x)$ for all real $x$, then
\begin{align}\label{th2-eq2}
\int_{-\infty}^\infty \big\{\Tlc(x) - L(x)\big\} \,\dx \ge \frac{1}{\delta} \left( \frac{\delta}{\lambda} -\frac{c}{2} -\frac{1}{e^{\delta^{-1}\lambda}-1}\right),
\end{align}
with equality if and only if $L(z) = L_{\delta^{-1}\lambda,c}(\delta z)$ for all $z \in \C$.

\medskip

\noindent {\rm (ii) (Majorant)} Let $c \le 1$. The inequality
\begin{align}\label{th3-eq1}
M_{\delta^{-1}\lambda,c}(\delta x)\geq \Tlc(x)  
\end{align}
holds for all real $x$. If $M$ is an entire function of exponential type $2\pi\delta$ satisfying $M(x)\ge \Tlc(x)$ for all real $x$, then
\begin{align}\label{th3-eq2}
\int_{-\infty}^\infty \big\{M(x) - \Tlc(x) \big\} \,\dx \ge \frac{1}{\delta} \left(\frac{1}{1-e^{-\delta^{-1}\lambda}} -\frac{\delta}{\lambda} -\frac{c}{2}\right),
\end{align}
with equality if and only if $M(z) = M_{\delta^{-1}\lambda,c}(\delta z)$ for all $z \in \C$.
\end{theorem}

\noindent {\it Remark 1.} For $c>1$ it turns out that the function $L_{\lambda,c}$ is the majorant of $\Tlc$ of type $2\pi$ (the usual change of variables gives the case $2\pi \delta$). This can be seen by noting that $\Tlc$ is the sum of $T_{\lambda,1}$ and of the step function that equals zero for negative $x$ and $c-1<0$ for positive $x$. The optimal majorant of $\Tlc$ can then be shown to be the sum of the majorant of $T_{\lambda,1}$ from Theorem \ref{tr-majorant-thm} and of the optimal majorant of the step function which can be obtained from \cite{V}. We omit the calculations since they are not needed in this paper.

\smallskip

\noindent {\it Remark 2.} Once we have solved the truncated problem we can easily obtain the solution of the odd problem. In fact, considering the odd function 
\begin{equation*}
\widetilde{T}_{\lambda, c}(x) = \sgn(x) \big(e^{-\lambda |x|} - c\big) = \Tlc (x) - \Tlc(-x),
\end{equation*}
we define
\begin{equation*}
\widetilde{K}_{\lambda, c}(z) = K_{\lambda,c}(z) - K_{\lambda,c}(-z).
\end{equation*}
When $c \leq e^{-\lambda}$ we use \eqref{th1-eq1} to get
\begin{align*}
\sin \pi x\, & \big\{ \widetilde{T}_{\lambda, c}(x) - \widetilde{K}_{\lambda, c}(x) \big\} \\
& = \sin \pi x\,  \{\Tlc (x) - K_{\lambda,c}(x)\} + \sin(-\pi x) \,  \{\Tlc (-x) - K_{\lambda,c}(-x)\} \geq 0
\end{align*}
for all real $x$. As we shall see in the proof of Theorem \ref{tr-ba-thm}, this is enough to conclude that $\widetilde{K}_{\lambda, c}$ is the unique best approximation of type $\pi$ of $\widetilde{T}_{\lambda, c}$. In a similar way we define
\begin{equation*}
\widetilde{L}_{\lambda, c}(z) = L_{\lambda,c}(z) - M_{\lambda,c}(-z)
\end{equation*}
and 
\begin{equation*}
\widetilde{M}_{\lambda, c}(z) = M_{\lambda,c}(z) - L_{\lambda,c}(-z). 
\end{equation*}
When $c \leq e^{-\lambda}$ we use \eqref{th2-eq1} and \eqref{th3-eq1} to get 
\begin{equation*}
\widetilde{L}_{\lambda, c}(x) \leq \widetilde{T}_{\lambda, c}(x) \leq \widetilde{M}_{\lambda, c}(x)
\end{equation*}
for all real $x$, and these functions interpolate $\widetilde{T}_{\lambda, c}$ at the integers. As we shall see in the proof of Theorem \ref{tr-majorant-thm}, this implies that they are the unique extremal one-sided approximations of type $2\pi$ for $\widetilde{T}_{\lambda, c}$. The general case of approximations of type $2 \pi \delta$ follows by a simple change of variables.

\subsection{General measures} Let $\nu$ be a nonnegative Borel measure on $(0,\infty)$. We specialize our shift parameter by putting $c = e^{-\lambda}$ and address here the extremal problem for the class of functions $\TT_{\nu}:\R \to \R \cup \{\infty\}$ given by
\begin{equation*}
\TT_{\nu}(x) = \int_0^{\infty}  T_{\lambda, e^{-\lambda}}(x)\,\dnu.
\end{equation*}
For technical reasons we introduce a new parameter $a>0$ and consider the family
\begin{equation*}
\TT_{\nu}(a; x) = \int_0^{\infty}  T_{a \lambda, e^{-\lambda}}(x)\,\dnu.
\end{equation*}
The problem of finding optimal approximations of general type $2\pi \delta$ for $\TT_{\nu}(x)$ is equivalent, via a simple scaling argument, to the problem of finding optimal approximations of type $2\pi$ for $\TT_{\nu}\big(\delta^{-1}; x\big)$.

\smallskip

For the best approximation problem, the minimal condition one should impose on the measure $\nu$ is the $\nu$-integrability of the right-hand side of \eqref{th1-eq2} in the case $c = e^{-\lambda}$. An analysis of the asymptotics of the right-hand side of \eqref{th1-eq2} when $\lambda \to 0$ and $\lambda \to \infty$ reveals that we must require
\begin{align}\label{min-nu-growth}
\int_0^\infty \frac{\lambda}{1+\lambda^2} \, \dnu<\infty.
\end{align}
Similarly, for the minorant problem we should have $\nu$-integrability of the right-hand side of \eqref{th2-eq2} and this condition is again given by \eqref{min-nu-growth}. For the majorant problem the requirement of $\nu$-integrability of the right-hand side of \eqref{th3-eq2} gives us the more restrictive condition
\begin{align}\label{maj-nu-growth}
\int_0^\infty \frac{\lambda}{1+\lambda} \, \dnu<\infty.
\end{align}
With either condition \eqref{min-nu-growth} or \eqref{maj-nu-growth}, for $x>0$, it is clear that $\lambda \mapsto \big(e^{-a\lambda x} - e^{-\lambda}\big)$ is $\nu$-integrable and that $x\mapsto \TT_{\nu}(a; x)$ is differentiable. Denoting $\tfrac{\partial}{\partial x} \TT_{\nu}(a; x)  = \TT'_{\nu}(a; x) $ we have, for $x>0$, 
\begin{equation*}
\TT'_{\nu}(a; x) = \int_0^{\infty} T'_{a \lambda, e^{-\lambda}}(x)\,\dnu = \int_0^{\infty} -a\lambda e^{-a \lambda x}\,\dnu.
\end{equation*}

\smallskip 

Inspired by \eqref{defKLM} we consider now the following functions, for $z \in \C$,
\begin{align}
\begin{split}\label{defKK}
\KK_{\nu}(a; z)& =\lim_{N\to \infty} \frac{\sin\pi z}{\pi} \left\{ \sum_{n=1}^N (-1)^{n} \frac{\TT_{\nu}(a;n)}{(z-n)}\right\} \\
& \qquad\qquad + \frac{\sin\pi z}{\pi z} \left\{ \int_0^{\infty}\left(-\frac{e^{-\lambda}}{2} + \sum_{n=1}^{\infty} (-1)^{n+1}  e^{- a \lambda n}\right)\dnu \right\},
\end{split}\\
\begin{split}\label{defLL}
\LL_{\nu}(a;z) & = \lim_{N\to \infty}  \frac{\sin^2\pi z}{\pi^2}\left\{ \sum_{n=1}^N \left(\frac{\TT_{\nu}(a;n)}{(z-n)^2} + \frac{\TT_{\nu}'(a;n)}{(z-n)}\right) \right\} \\
& \qquad \qquad + \frac{\sin^2\pi z}{\pi^2 z}\left\{ \int_0^{\infty} \left(-e^{-\lambda } + \sum_{n=1}^\infty  a\,\lambda \,e^{-a \lambda n}\right) \dnu \right\},
\end{split}\\\label{defMM}
\MM_\nu(a;z) &= \LL_\nu(a; z) + \left(\frac{\sin\pi z}{\pi z}\right)^2 \int_0^\infty \big(1-e^{-\lambda}\big) \,\dnu.
\end{align}
Under \eqref{min-nu-growth} we shall prove that the sequence of entire functions in \eqref{defKK} converges uniformly on compact subsets of $\C$ and thus the limit does indeed define an entire function that will be proved to be of exponential type $\pi$. A similar statement holds for \eqref{defLL}, where the limiting entire function  will be proved to be of exponential type $2\pi$. In the case of \eqref{defMM}, the more restrictive condition \eqref{maj-nu-growth} on the measure $\nu$ guarantees the convergence of the last integral. 

\begin{theorem}[Optimal two-sided approximation] \label{ba-nu-thm} Assume that $\nu$ satisfies \eqref{min-nu-growth} and let $\delta\geq1$. The inequality
\begin{align*}
\sin\pi \delta x\,\big\{\TT_\nu(x) - \KK_\nu\big(\delta^{-1};\delta x\big) \big\}\ge 0
\end{align*}
holds for all real $x$. If $\KK$ is an entire function of type $\pi\delta$, then
\begin{align*}
\int_{-\infty}^\infty \big|\TT_\nu(x)- \KK(x)\big|\,\dx \ge \int_0^\infty  \frac{1}{\delta} \left( \frac{1-e^{-\delta^{-1}\lambda}}{\delta^{-1}\lambda \big(1+e^{-\delta^{-1}\lambda}\big)} - \frac{e^{-\lambda}}{2}\right) \dnu,
\end{align*}
with equality if and only if $\KK(z) = \KK_\nu\big(\delta^{-1};\delta z\big)$ for all $z \in \C$.
\end{theorem}

\smallskip

\begin{theorem}[Optimal one-sided approximations] \label{maj-nu-thm}.

\smallskip

\noindent {\rm (i) (Minorant)} Assume that $\nu$ satisfies \eqref{min-nu-growth} and let $\delta\geq1$. The inequality
\begin{align*}
\LL_\nu\big(\delta^{-1};\delta x\big)\le \TT_\nu(x)
\end{align*}
holds for all real $x$. If $\LL$ is an entire function of exponential type $2\pi\delta$ satisfying $\LL(x)\le \TT_\nu(x)$ for all real $x$, then
\begin{align*}
\int_{-\infty}^\infty \big\{\TT_\nu(x)-\LL(x)\big\} \,\dx \ge \int_0^\infty  \frac{1}{\delta} \left( \frac{\delta}{\lambda} -\frac{e^{-\lambda}}{2} -\frac{1}{e^{\delta^{-1}\lambda}-1}\right)\dnu,
\end{align*}
with equality if and only if $\LL(z) = \LL_\nu\big(\delta^{-1};\delta z\big)$ for all $z \in \C$. 

\smallskip

\noindent {\rm (ii) (Majorant)} Assume that $\nu$ satisfies \eqref{maj-nu-growth} and let $\delta>0$. The inequality
\begin{align*}
\MM_\nu\big(\delta^{-1};\delta x\big)\ge \TT_\nu(x)
\end{align*}
holds for all real $x$. If $\MM$ is an entire function of exponential type $2\pi\delta$ satisfying $\MM(x)\ge \TT_\nu(x)$ for all real $x$, then
\begin{align*}
\int_{-\infty}^\infty \big\{\MM(x) - \TT_\nu(x)\big\}\, \dx \ge \int_0^\infty \frac{1}{\delta}\left( \frac{1}{1-e^{-\delta^{-1}\lambda}} -\frac{\delta}{\lambda} -\frac{e^{-\lambda}}{2}\right) \dnu,
\end{align*}
with equality if and only if $\MM(z) = \MM_\nu\big(\delta^{-1};\delta z\big)$ for all $z \in \C$. 
\end{theorem}

The solutions to the corresponding extremal problems for the family of odd functions given by $\widetilde{\TT}_{\nu}(x) = \TT_{\nu}(x) -\TT_{\nu}(-x)$ follow in the same way as in Remark 2. For the best approximation we need  \eqref{min-nu-growth} and exponential type at least $\pi$. For the one-sided approximations we need the stronger assumption \eqref{maj-nu-growth} and exponential type at least $2\pi$.

\subsection{Examples} Theorems \ref{ba-nu-thm} and \ref{maj-nu-thm} allow measures $\nu$ with a stronger singularity at the origin than the Gaussian subordination framework of \cite{CL} permits. For the best approximation (of type at least $\pi$) and minorant (of type at least $2\pi$) we can consider for instance the measures ${\rm d}{\nu_{\alpha}} (\lambda) = \lambda^{-\alpha} \, \dl$, with $0< \alpha<2$, for they satisfy \eqref{min-nu-growth}.  We then get truncated shifted power functions
\begin{equation}\label{exem1}
\TT_{\nu_{\alpha}}(x) = \left\{
\begin{array}{rl}
\Gamma(1- \alpha) \big\{ |x|^{\alpha -1} -1\big\} & {\rm if} \ \ x>0,\\
0 & {\rm if} \ \ x<0,
\end{array}
\right.
\end{equation}
if $\alpha \neq 1$, and in the case $\alpha = 1$ we get the truncated logarithm
\begin{equation*}
\TT_{\nu_{-1}}(x) = \left\{
\begin{array}{rl}
-\log x & {\rm if} \ \ x>0,\\
0 & {\rm if} \ \ x<0.
\end{array}
\right.
\end{equation*}
The measure $\nu_{\alpha}$ satisfies \eqref{maj-nu-growth} only if $1< \alpha <2$, and in this case we also get majorants (of any exponential type) for the functions in \eqref{exem1}. For the odd versions we can still consider $0< \alpha <2$ for the best approximation problem  (of type at least $\pi$) and $1< \alpha <2$ for the one-sided approximations (of type at least $2\pi$). Notice that, in the framework of \cite{CL}, one gets instead the best approximation and minorants of any exponential type for the truncated power functions $|x|^{\beta}$ for $-1 < \beta < 0$.

\section{Proofs of Theorems \ref{tr-ba-thm} and \ref{tr-majorant-thm}}
\subsection{Preliminaries} We define functions $b:\R\to\R$ and $B:\R \to \R$ by 
\begin{equation}\label{def-b}
b(w) = \frac{1}{1+e^w},
\end{equation}
and
\begin{equation}\label{def-B}
B(w) = \left\{
\begin{array}{rl}
\displaystyle \frac{w}{1-e^{-w}} & {\rm if}\ \ w \neq 0,\\
1& {\rm if}\ \ w = 0.\\
\end{array}
\right.
\end{equation}

We define for $\lambda\ge 0$ the function $E_\lambda: \R \to \R$ by
\begin{equation*}
E_{\lambda}(x) = \left\{
\begin{array}{cc}
e^{-\lambda x} & \textrm{if} \ \ x>0,\\
\frac{1}{2} & \textrm{if} \ \ x=0,\\
0 & \textrm{if} \ \ x<0.
\end{array}
\right.
\end{equation*}

It is reasonable to expect that for $\lambda>0$ the optimal functions for $\Tlc$ are related to the optimal functions of $\El$ and $cE_0$ (that can be found in \cite{GV} and \cite{V}, respectively). We define therefore
\begin{align*}
K_{\lambda}(z) &= \frac{\sin\pi z}{\pi} \sum_{n=1}^\infty (-1)^n\left(\frac{E_\lambda(n)}{z-n}-\frac{e^{-\lambda n}}{z}\right), \\
M_{\lambda}(z) &= \frac{\sin^2\pi z}{\pi^2} \left\{\sum_{n=1}^\infty \left(\frac{E_\lambda(n)}{(z-n)^2} + \frac{E_\lambda'(n)}{(z-n)}-\frac{E_\lambda'(n)}{z}\right) + \frac{1}{z^2} \right\},
\end{align*}
for $\lambda>0$, and
\begin{align*}
K_0(z) &= \frac{\sin\pi z}{\pi} \left\{ \sum_{n=1}^\infty \frac{(-1)^n}{z-n} +\frac{1}{2z}\right\},\\
M_{0}(z) &= \frac{\sin^2\pi z}{\pi^2} \left\{\sum_{n=1}^\infty \frac{1}{(z-n)^2} + \frac{1}{z} + \frac{1}{z^2} \right\}.
\end{align*}

\begin{lemma} [Integral representation] Let $\lambda\ge 0$. For $b$ defined in \eqref{def-b} we have
\begin{align}\label{kl-rep}
K_\lambda(x) - E_\lambda(x) &=
\begin{cases}
\displaystyle \frac{\sin\pi x}{\pi } \int_0^\infty \{b(\lambda+ w) - b(\lambda)\} \,e^{x w} \,\dw&\text{ if }x<0, \\[1.5ex]
\displaystyle \frac{\sin \pi x}{\pi} \int_{-\infty}^0 \{b(\lambda) - b(\lambda+w)\}\,e^{x w}\, \dw&\text{ if }x>0,
\end{cases}
\end{align}
and for $B$ defined in \eqref{def-B} we have
\begin{align}\label{ml-rep}
M_{\lambda}(x)-E_\lambda(x) &=  \begin{cases}\displaystyle \frac{\sin^2\pi x}{\pi^2} \int_0^{\infty}\{ B(\lambda + w) - B(\lambda)\}\,e^{xw}\,\dw & \text{ if }x<0,\\[1.5ex]
\displaystyle \frac{\sin^2\pi x}{\pi^2} \int_{-\infty}^{0}\{ B(\lambda) - B(\lambda + w)\}\,e^{xw}\,\dw & \text{ if }x>0.
\end{cases}
\end{align}
\end{lemma}

\begin{proof} Note first that
\begin{align*}
b(w) = \begin{cases}
\displaystyle \sum_{n=0}^\infty (-1)^n e^{nw}&\text{ if }w<0,\\[1.5ex]
\displaystyle -\sum_{n=1}^\infty (-1)^n e^{-nw} &\text{ if }w>0.
\end{cases}
\end{align*}
We then have for $x<0$ and $\lambda\ge 0$
\begin{align*}
\sum_{n=1}^N (-1)^n \frac{e^{-\lambda n}}{x-n} = -\int_\lambda^\infty e^{-\lambda x} \sum_{n=1}^N (-1)^n \,e^{w(x-n)} \,\dw.
\end{align*}
An application of dominated convergence gives
\begin{align*}
\sum_{n=1}^\infty (-1)^n \frac{e^{-\lambda n}}{x-n} = \int_0^\infty b(\lambda+w) \, e^{w x}\, \dw,
\end{align*}
and an expansion of $x^{-1}$ as a Laplace integral gives \eqref{kl-rep} for $x<0$. To prove the representation for $x>0$ we split after a change of variables
\begin{align*}
\int_{-\infty}^0 b(\lambda+w) \,e^{w x}\, \dw = \int_{-\infty}^0 b(w) \,e^{(w-\lambda)x} \,\dw + \int_0^\lambda b(w) \,e^{(w-\lambda)x} \,\dw,
\end{align*}
and note that the second integral is not present if $\lambda=0$. Dominated convergence gives for the first integral
\begin{align*}
\int_{-\infty}^0 b(w) \,e^{(w-\lambda)x}\, \dw = e^{-\lambda x} \sum_{n=0}^\infty \frac{(-1)^n }{x+n},
\end{align*}
while the second integral satisfies
\begin{align*}
\int_0^\lambda b(w) \,e^{(w-\lambda)x}\, \dw = e^{-\lambda x} \sum_{n=1}^\infty (-1)^{n+1}\left\{\frac{e^{\lambda(x-n)}}{x-n}- \frac{1}{x-n}\right\}.
\end{align*}
Combining these identities with
\begin{align*}
\frac{\pi}{\sin\pi x} = \sum_{n=-\infty}^\infty \frac{(-1)^n}{x-n}
\end{align*}
gives \eqref{kl-rep} when $x>0$.

\smallskip

For \eqref{ml-rep} the calculations are analogous, using the fact that
\begin{align*}
B(w) = \begin{cases}
\displaystyle -w\sum_{n=1}^\infty  e^{nw}&\text{ if }w<0,\\[3.0ex]
\displaystyle w\sum_{n=0}^\infty  e^{-nw} &\text{ if }w>0,
\end{cases}
\end{align*}
and
\begin{align*}
\left(\frac{\pi}{\sin\pi x}\right)^2 = \sum_{n=-\infty}^\infty \frac{1}{(x-n)^2},
\end{align*}
for all $x$.
\end{proof}

Several of the inequalities that are required for the proofs in this section rest on the fact $b$ and $B$ are convex functions. It is well known that a convex function that has value zero at the origin is superadditive on the positive reals. For completeness we include a short proof of this fact in the next lemma.

\begin{lemma}\label{superadditive} Let $f$ be a continuous convex function defined on $[0,\infty)$ with $f(0) =0$. Then $f$ satisfies
\begin{align*}
f(x+y)\ge f(x)+ f(y)
\end{align*}
for all nonnegative $x$ and $y$.
\end{lemma}

\begin{proof}
Since $f(0)=0$ we obtain from the convexity of $f$, for $0\leq t \leq 1$ and $z>0$, that $f(tz)\le tf(z)$. Applying this with $z=x+y$ and  $t=y(x+y)^{-1}$ as well as $z=x+y$ and $t=x(x+y)^{-1}$ gives for all positive $x$ and $y$ the inequality $f(x+y)\ge f(x)+f(y)$.
\end{proof}

\subsection{Proof of Theorem \ref{tr-ba-thm}} Observe first that, via a simple scaling argument, it suffices to prove the result for $\delta =1$, and we henceforth restrict ourselves to this case. 

\subsubsection{Inequalities} We show \eqref{th1-eq1}. For $x<0$ the integral representation \eqref{kl-rep} gives us
\begin{align}\label{klc-rep-in-negx}
\begin{split}
K_{\lambda,c}(x)  &=\frac{\sin\pi x}{\pi} \int_0^\infty \big\{b(\lambda+ w) - b(\lambda) -c(b(w) - b(0))\big\} \,e^{x w}\, \dw.
\end{split}
\end{align}
We consider the endpoint value $c = e^{-\lambda}$ and define
\begin{align*}
g(\lambda,w) = e^\lambda b(\lambda+ w) - e^\lambda b(\lambda) -b(w) + b(0).
\end{align*}
A direct calculation gives
\begin{align*}
\frac{\partial g}{\partial \lambda}(\lambda,w) = e^\lambda\left( \frac{1}{(1+e^{\lambda+w})^2} -\frac{1}{(1+e^\lambda)^2}\right)\le 0
\end{align*}
for $\lambda \ge 0$ and $w\ge 0$. Since $g(0,w) = 0$, it follows that 
\begin{align*}
e^{-\lambda} g(\lambda,w) \le 0,
\end{align*}
and inserting this into \eqref{klc-rep-in-negx} we obtain
\begin{align}
\sin\pi x\, \big\{T_{\lambda,e^{-\lambda}}(x) - K_{\lambda,e^{-\lambda}}(x)\big\} = -\sin\pi x\, K_{\lambda,e^{-\lambda}}(x) \ge 0
\end{align}
for any $x<0$.

\smallskip

Let us consider now the case $x>0$. From \eqref{kl-rep} we obtain 
\begin{align}\label{klc-rep-in-posx}
\begin{split}
K_{\lambda,c}(x) - T_{\lambda,c}(x)  &=\frac{\sin\pi x}{\pi} \int_{-\infty}^0 \!\!\big\{b(\lambda) - b(\lambda+ w)  -c(b(0) - b(w))\big\} \,e^{x w} \,\dw
\end{split}
\end{align}
for any $c \in \R$. To analyze the sign of this difference for $c = e^{-\lambda}$ we set
\[
h(\lambda,w) = -g(\lambda,w).
\]
An analogous calculation shows that $h(\lambda,w)\le 0$ for $\lambda\ge 0$ and $w\le 0$, and inserting this into \eqref{klc-rep-in-posx} we get 
\begin{align}
\sin\pi x \,\big\{T_{\lambda, e^{-\lambda}}(x) - K_{\lambda,e^{-\lambda}}(x)\big\} \ge 0
\end{align}
for any $x >0$.

\smallskip

In the general case $c\le e^{-\lambda}$ we note that
\begin{align*}
\Tlc(x) = T_{\lambda,e^{-\lambda}}(x) + (e^{-\lambda} -c) E_0(x).
\end{align*}
Since $E_0 = \tfrac{1}{2}(\sgn(x)+x)$, from \cite[Lemma 1]{V} we have for all real $x$ that 
\begin{align}\label{k0-ineq}
\sin\pi x\,\big\{E_0(x) - K_0(x)\big\} \ge0.
\end{align}
The identity $K_{\lambda,c} = K_{\lambda,e^{-\lambda}} + (e^{-\lambda} -c)K_0$  implies with \eqref{k0-ineq} that 
\begin{align*}
\sin\pi x\,\big\{\Tlc(x) - K_{\lambda,c}(x)\big\} \ge 0
\end{align*}
for all real $x$. This concludes the proof of \eqref{th1-eq1}.

\subsubsection{Integral evaluation}  Recall that 
\begin{align*}
\sgn(\sin \pi x)\big\{\Tlc(x) - K_{\lambda,c}(x)\big\} & =  \sgn(\sin \pi x)\big\{E_{\lambda}(x) - K_{\lambda}(x)\big\} \\
& =  \sgn(\sin \pi x)\big\{E_{0}(x) - K_{0}(x)\big\} \geq 0.
\end{align*}
Thus
\begin{align*}
 \big|\Tlc(x) - K_{\lambda,c}(x)\big| &= \big|\big(E_{\lambda}(x) - K_{\lambda}(x)\big) - c \big(E_{0}(x) - K_{0}(x)\big)\big| \\
 &=  \big|E_{\lambda}(x) - K_{\lambda}(x)\big| - c \,\big|E_{0}(x) - K_{0}(x)\big|,
\end{align*}
and we arrive at 
\begin{equation*}
\int_{-\infty}^{\infty} \big|\Tlc(x) - K_{\lambda,c}(x)\big|\, \dx = \int_{-\infty}^{\infty} \big|E_{\lambda}(x) - K_{\lambda}(x)\big|\, \dx - c \int_{-\infty}^{\infty} \big|E_{0}(x) - K_{0}(x)\big|\, \dx.
\end{equation*}
From \cite[Theorem 4]{V} we have that 
\begin{equation}\label{EV-1}
\int_{-\infty}^{\infty} |E_{0}(x) - K_{0}(x)|\, \dx = \frac{1}{2}.
\end{equation}
The $L^1$-norm of $E_\lambda - K_\lambda$, for $\lambda>0$, was calculated in \cite[Lemma 3.5]{HZ}. We give a short alternative argument here. Recall the fact that $K_{\lambda}$ is integrable, bounded on $\R$, and has exponential type $\pi$, so by the Paley-Wiener theorem its Fourier transform is a continuous function supported on the interval $[-\tfrac12, \tfrac12]$. Observe that $\sgn(\sin \pi x)$ is a normalized function of bounded variation on $[0,2]$, so the partial sums of the Fourier expansion 
\begin{equation}\label{FTsgn}
\sgn(\sin \pi x) =  \frac{i}{\pi} \lim_{N\to\infty}\sum_{n = -N}^{N-1} \frac{1}{(n+\tfrac12)}e^{-2\pi i (n + \frac12)x}
\end{equation}
are uniformly bounded (and converge at every point $x$). Inserting \eqref{FTsgn} in the identity
\[
\int_{-\infty}^\infty \big|E_{\lambda}(x) - K_{\lambda}(x)\big|\, \dx=\left|\int_{-\infty}^\infty \sgn(\sin \pi x)\big\{E_{\lambda}(x) - K_{\lambda}(x)\big\}\,\dx\right|,
\]
interchanging limit and integration, and using $\widehat{E_{\lambda}}(y)  = (2\pi i y + \lambda)^{-1}$, we get
\begin{align*}
\int_{-\infty}^{\infty} & \big|E_{\lambda}(x) - K_{\lambda}(x)\big|\, \dx = \left|\frac{i}{\pi} \sum_{n\in \Z} \frac{1}{\big(n + \tfrac12\big)} \frac{1}{\big(2\pi i \big(n + \tfrac12\big) + \lambda\big)} \right|.
\end{align*}
Combining the summands $n=k$ and $n=-k-1$, where $k$ is a nonnegative integer, gives us
\begin{align*}
\int_{-\infty}^{\infty} & \big|E_{\lambda}(x) - K_{\lambda}(x)\big|\, \dx =\sum_{k\geq 0} \frac{4}{\lambda^2 + 4 \pi^2 \big(k + \tfrac12\big)^2}.
\end{align*}
Poisson summation implies that
\begin{align*}
\sum_{k\geq 0} \frac{4}{\lambda^2 + 4 \pi^2 \big(k + \tfrac12\big)^2} = \sum_{n \in \Z} \frac{2}{\lambda^2 + 4 \pi^2 (n + \tfrac12)^2} &=\frac{1}{\lambda}  \,\sum_{k \in \Z} (-1)^k e^{-\lambda |k|},
\end{align*}
and evaluation of the series on the right gives
\begin{align}\label{EV-2}
\begin{split}
\int_{-\infty}^{\infty} & \big|E_{\lambda}(x) - K_{\lambda}(x)\big|\, \dx = \frac{1- e^{-\lambda}}{\lambda \big(1 + e^{-\lambda}\big)}.
\end{split}
\end{align}
From \eqref{EV-1} and \eqref{EV-2} we obtain
\begin{equation*}
\int_{-\infty}^\infty \big|\Tlc(x)- K_{\lambda,c}(x)\big| \, \dx = \frac{1 - e^{-\lambda}}{ \lambda \big(1 + e^{-\lambda}\big)} - \frac{c}{2}.
\end{equation*}

\subsubsection{Optimality and uniqueness} Let $K$ be an entire function of exponential type $\pi$ such that $\Tlc - K$ is integrable. It follows that $K-K_{\lambda,c}$ is integrable and, since it has exponential type $\pi$, its Fourier transform is a continuous function supported in $[-\tfrac12,\tfrac12]$. From \eqref{FTsgn} and dominated convergence we have
\begin{align*}
\int_{-\infty}^\infty \sgn(\sin\pi x)\big\{K(x) - K_\lambda(x)\big\}\, \dx =0.
\end{align*}
Since $\sgn(\sin\pi x)\big\{\Tlc(x) - K_{\lambda,c}(x)\big\}\ge 0$, we obtain 
\begin{align}\label{extremal-ineq}
\begin{split}
\int_{-\infty}^\infty \big|\Tlc(x) - K(x)\big|\, \dx &\ge \left|\int_{-\infty}^\infty \sgn(\sin\pi x) \big\{\Tlc(x) - K(x)\big\} \,\dx\right|\\
&= \int_{-\infty}^\infty \sgn(\sin\pi x)\big\{\Tlc(x) - K_{\lambda,c}(x)\big\}\, \dx.
\end{split}
\end{align}
If there is equality in \eqref{extremal-ineq} we must have $K(n) = \Tlc(n) = K_{\lambda,c}(n)$ for all $n \in \Z/\{0\}$. The Fourier transform of $K-K_{\lambda,c}$ is in $L^\infty(\RR)$ and has compact support, hence $K-K_{\lambda,c}$ is in $L^2(\RR)$. It follows \cite[Chapter XVI, equation (7.19)]{Z} that
\[
K(z) - K_{\lambda,c}(z) = \alpha \frac{\sin\pi z}{\pi z}
\]
for some $\alpha\in\CC$, and since the left-hand side is integrable we must have $\alpha=0$. This concludes the proof of Theorem \ref{tr-ba-thm}.

\subsection{Proof of Theorem \ref{tr-majorant-thm} - Majorant} Note again that via a scaling argument, it suffices to prove the result for $\delta =1$, and this shall be henceforth assumed for the one-sided approximations. In the majorant case recall that $c \leq 1$. 

\subsubsection{Inequalities.} Let us consider first the case $x<0$. It follows from \eqref{ml-rep} that
\begin{align}\label{mlb-rep}
\begin{split}
M_{\lambda,c}(x) &= M_{\lambda}(x)  - cM_{0}(x)  \\
&= \frac{\sin^2\pi x}{\pi^2} \int_0^{\infty}\big\{ B(\lambda + w) - B(\lambda) - c(B(w) - B(0))\big\}\,e^{xw}\,\dw.
\end{split}
\end{align}
We will show that the integrand in \eqref{mlb-rep} is nonnegative. Using the fact that $x\coth x-1\ge 0$ for all real $x$ we obtain 
\begin{align}\label{second-derivative}
B''(w) = \frac{we^w}{(e^w-1)^2} \left(\coth(\tfrac{w}{2})-\tfrac{2}{w}\right)\ge 0
\end{align}
for all real $w$, i.e. $B$ is convex. The function $f$ defined by 
\[
f(w) = B(w)-B(0)
\]
is therefore superadditive by Lemma \ref{superadditive} for nonnegative $w$. This means that $f(\lambda+w)\ge f(\lambda) + f(w)$, which implies 
\[
B(w+\lambda) - B(\lambda)\ge B(w) - B(0).
\]
Since $B(w)-B(0)\ge 0$, we obtain for positive $w$ and $\lambda$, and any $c\le 1$,
\begin{align*}
B(w+\lambda) - B(\lambda) - c(B(w)-B(0)) \ge 0\,,
\end{align*}
and inserting this in \eqref{mlb-rep} shows that $M_{\lambda,c}(x)\ge 0$ for $x<0$.

\smallskip

In the range $x>0$ we need to consider $w<0$ and $\lambda>0$. We define $f_w$ by
\begin{align*}
f_w(x) = B(x+w) - B(w).
\end{align*}
We note that since $B$ is increasing on $\RR$, the function $f_w$ is nonnegative for $x\ge 0$. Furthermore, $f_w(0) =0$ and $f_w$ is convex by \eqref{second-derivative}. By Lemma \ref{superadditive}, $f_w$ is superadditive on the positive half-axis, hence for $\lambda\ge 0$ and $w\le 0$,
\begin{align*}
f_w(\lambda-w) \ge f_w(\lambda) + f_w(-w).
\end{align*}
The definition of $f_w$ gives
\begin{align*}
B(\lambda) - B(\lambda + w) \ge  B(0) -B(w).
\end{align*}
Since $B(0) - B(w)\ge 0$ for negative $w$, we obtain for every $c\le 1$,
\begin{align}\label{i2-pf}
B(\lambda) - B(\lambda+w) - c(B(0) - B(w))\ge 0.
\end{align}
Inequality \eqref{i2-pf} and identity \eqref{ml-rep} imply that 
\begin{align*}
M_{\lambda,c}(x) - \Tlc(x) \ge 0
\end{align*}
for $x>0$.

\subsubsection{Integral evaluation} From \cite[Theorem 6]{GV} we have $M_\lambda\in L^1(\R)\cap L^2(\R)$, 
and from \cite[Theorem 8]{V} we have $M_0-E_0\in L^1(\R)$. Hence it follows that
\begin{align*}
M_{\lambda,c}  - \Tlc\in L^1(\R).
\end{align*}
By construction $M_\lambda(n) = e^{-\lambda n}$ for nonnegative integers $n$ and $M_\lambda(n) =0$ for negative integers $n$. By a result of Polya and Plancherel \cite{PP}, $M_\lambda$ has bounded variation, and we can apply Poisson summation and the Paley-Wiener theorem to get
\begin{align*}
\int_{-\infty}^\infty \big\{M_\lambda(x) - E_\lambda(x)\big\} \, \dx = \sum_{n=0}^\infty e^{-\lambda n} -\frac{1}{\lambda}= \frac{1}{1-e^{-\lambda}}-\frac{1}{\lambda}.
\end{align*}
From \cite[Theorem 8]{V} it follows that
\begin{align*}
\int_{-\infty}^\infty \big\{M_0(x) - E_0(x)\big\} \,\dx =\frac12.
\end{align*}
The identity $M_{\lambda, c} = M_\lambda - c M_0$ gives therefore
\begin{align}\label{mlb-integral-value}
\int_{-\infty}^\infty \big\{M_{\lambda,c}(x) - \Tlc(x)\big\} \,\dx = \frac{1}{1-e^{-\lambda}} -\frac1\lambda -\frac{c}{2}.
\end{align}

\subsubsection{Optimality and uniqueness} Let $M\ge \Tlc$ be an entire function of exponential type $2\pi$. We may assume that $M - \Tlc$ is integrable. Then $M-M_{\lambda,c}$ is integrable and entire of exponential type $2\pi$, hence this difference is of bounded variation by \cite{PP}. By construction we have $M(n)\ge \Tlc(n) = M_{\lambda,c}(n)$ for all integers $n$ (at $n=0$ we should take the upper limit). Poisson summation and the Paley-Wiener theorem give us
\begin{align}\label{M-Mlb-integral}
\int_{-\infty}^\infty \big\{M(x) - M_{\lambda,c}(x)\big\} \, \dx = \sum_{n\in\ZZ} \big\{M(n) - M_{\lambda,c}(n)\big\} \ge 0,
\end{align}
which implies
\begin{align}\label{optimality-integral}
\int_{-\infty}^\infty \big\{M(x) - \Tlc(x)\big\} \,\dx \ge \int_{-\infty}^\infty \big\{M_{\lambda,c}(x) -\Tlc(x)\big\}\, \dx.
\end{align}
If $M$ is such that we have equality in \eqref{optimality-integral}, then we must also have equality in \eqref{M-Mlb-integral}. Thus $M(n) = M_{\lambda,c}(n)$ for all integers $n$. At every non-zero integer we have necessarily
\begin{align*}
M'(n) = \Tlc'(n) = M_{\lambda,c}'(n),
\end{align*}
and \cite[Theorem 9]{V} implies that
\begin{align}\label{finaldif}
M(z) - M_{\lambda,c}(z) = \big( M'(0) - M_{\lambda,c}'(0)\big)\frac{\sin^2\pi z}{\pi^2 z}.
\end{align}
Since the function on the left-hand side of \eqref{finaldif} is integrable, the derivatives at the origin agree, and we obtain $M = M_{\lambda,c}$.

\subsection{Proof of Theorem \ref{tr-majorant-thm} - Minorant} Recall that, by scaling, we have reduced our study to the particular value $\delta =1$, and in the minorant case we have $c\le e^{-\lambda}$.

\subsubsection{Inequalities} We note from the definition of $L_{\lambda,c}$ and $M_{\lambda,c}$ that
\begin{align}\label{relation-llb-mlb}
L_{\lambda,c}(z) = M_{\lambda,c}(z) - (1-c)\left(\frac{ \sin\pi z}{\pi z}\right)^2,
\end{align}
and from \eqref{mlb-rep} we obtain for $x<0$
\begin{align}\label{llb-int-rep}
L_{\lambda,c}(x) \!=\! \frac{\sin^2\pi x}{\pi^2} \!\int_0^{\infty}\!\!\big\{ B(\lambda + w) \!-\! B(\lambda) \!-\! cB(w) \!+\! cB(0) \!-\!w(1-c)\big\}e^{xw}\dw.
\end{align}
We consider first the case $c= e^{-\lambda}$. Define
\begin{align}\label{llb-integrand-def}
g(\lambda,w) =  e^\lambda B(\lambda + w) - e^\lambda B(\lambda) - B(w) + 1 -w(e^\lambda -1).
\end{align}
We have
\begin{align}\label{partial1}
g(0,w) = 0
\end{align}
and
\[
\frac{\partial g}{\partial \lambda}(\lambda ,w) = e^\lambda\left(\frac{1}{e^{\lambda+w} -1}  - \frac{\lambda+w}{(e^{\lambda+w}-1)^2} + \frac{1+\lambda-e^\lambda}{(e^\lambda-1)^2}\right).
\]
Since 
$$x\mapsto \frac{1}{(e^x-1)} - \frac{x}{(e^x-1)^2}$$
is a decreasing function, we obtain for positive $\lambda$ and $w$
\begin{align}\label{partial2}
\frac{\partial g}{\partial \lambda}(\lambda ,w) \le e^\lambda \left(\frac{1}{e^{\lambda} -1} -\frac{\lambda}{(e^{\lambda}-1)^2} + \frac{1+\lambda-e^\lambda}{(e^\lambda-1)^2} \right) =0.
\end{align}
It follows from \eqref{partial1} and \eqref{partial2} that
\begin{align}\label{llb-integrand-nonpos}
g(\lambda,w)\le 0
\end{align}
for $\lambda\ge 0$ and $w\ge 0$. Identity \eqref{llb-int-rep} and inequality \eqref{llb-integrand-nonpos} imply that 
\begin{align*}
L_{\lambda,e^{-\lambda}}(x) \le 0
\end{align*}
for $x<0$.

\smallskip

From \eqref{ml-rep} and \eqref{relation-llb-mlb} we obtain for $x>0$
\begin{align*}
\begin{split}
&L_{\lambda,c}(x) - \Tlc(x) \\
&\,=\frac{\sin^2\pi x}{\pi^2 }\int_{-\infty}^0\{  B(\lambda) - B(\lambda+w) - c B(0) + cB(w)+w(1-c)\}\, e^{x w} \,\dw.
\end{split}
\end{align*}
We note that the integrand for $c = e^{-\lambda}$ multiplied by $e^\lambda$ is
\begin{align*}
h(\lambda,w) = -g(\lambda,w),
\end{align*}
with $g$ defined in \eqref{llb-integrand-def}. Analogously to the calculation for $\lambda\ge 0$ and $w\ge 0$ one can establish that
\begin{align*}
h(\lambda,w)\le 0
\end{align*}
for $\lambda\ge 0$ and $w\le 0$. It follows that
\begin{align*}
L_{\lambda,e^{-\lambda}}(x) - T_{\lambda,e^{-\lambda}}(x) \le 0
\end{align*}
for $x>0$. 

\smallskip

To show that $L_{\lambda,c}$ is a minorant of $\Tlc$ for $c\le e^{-\lambda}$, we note that
\begin{align*}
\Tlc(x) = T_{\lambda,e^{-\lambda}}(x) + \big(e^{-\lambda} -c\big)E_0(x),
\end{align*}
and the coefficient of $E_0$ is nonnegative. Hence, since
\begin{align*}
L_{\lambda,c}(x) = L_{\lambda,e^{-\lambda}}(x) +\big(e^{-\lambda}-c\big) L_0(x),
\end{align*}
where $L_0$ is the minorant of $E_0$, for $c\le e^{-\lambda}$ the sum of the two minorants is a minorant of $\Tlc$. 

\subsubsection{Integral evaluation} Identity \eqref{relation-llb-mlb} and expression \eqref{mlb-integral-value} imply
\begin{align*}
\int_{-\infty}^\infty \big\{\Tlc(x) - L_{\lambda,c}(x) \big\}\,\dx = \frac1\lambda -\frac{c}{2} -\frac{1}{e^\lambda-1}.
\end{align*}

\subsubsection{Optimality and uniqueness} This is analogous to the majorant part and we omit the details. This concludes the proof of Theorem \ref{tr-majorant-thm}.

\section{Growth estimates via contour integration}

In the proof of Theorem \ref{ba-nu-thm} (and analogously for Theorem \ref{maj-nu-thm}) we are faced with the problem of proving that functions of the form
\[
z\mapsto \frac{\sin\pi z}{\pi} \left\{ \sum_{n=1}^\infty (-1)^{n} \frac{\TT_{\nu}\big(\delta^{-1};n\big)}{(z-n)}\right\}
\]
are entire functions of finite exponential type. \textsl{A priori} it is not even clear that the series converges for all complex $z$. In this section we investigate, in a general setting, the growth behavior of certain series defined by a base function $\Phi(z)$ with mild decay properties. We rely on the framework introduced in \cite[Section 2]{CV2} and \cite[Section 3]{CV3}, where the corresponding problem for even functions was treated. In order to facilitate the references, we try to keep the notation of \cite{CV2, CV3} as close as possible. The idea, briefly, consists on noting that for fixed non-integer $z\in\C$ the meromorphic function
\[
w\mapsto \frac{\sin\pi z}{\sin\pi w}\, \frac{\Phi(w)}{(z-w)}
\]
has poles at $w=n$ ($n$ integer) with residue $\mc{A}(n,\Phi,z)$ given by
\[
\mc{A}(n,\Phi,z) = \frac{\sin\pi z}{\pi} (-1)^n \frac{\Phi(n)}{(z-n)}.
\]
Hence any integral along a closed contour that contains $N$ of these poles equals to a partial sum of the interpolation series. This leads to an identity that relates the tail of the interpolation series with $M< n\le N$ to a contour integral along the sides of a rectangle contained in $M+1/2 \le \Re w\le  N+1/2$. This integral is then shown to be uniformly convergent on compact sets, which implies that the interpolation series defines an entire function.

\smallskip

Throughout this section we let $\mc{R} = \{ z \in \C; 0<\Re(z)\}$ denote the open half plane and we let $k$ be $1$ or $2$ (the choice $k=1$ will be used in the two-sided approximation problem, while the choice $k=2$ will be used in the one-sided approximation problem). We let $\Phi(z)$ be a function that is analytic on $\mc{R}$ and satisfies the following three properties: if $0 < a < b < \infty$ then
\begin{equation}\label{cv10}
\lim_{y\rightarrow \pm\infty} e^{-k\pi |y|} \int_a^b \left|\frac{\Phi(x+iy)}{x+iy}\right|\, \dx = 0,
\end{equation}
if $0 < \eta < \infty$ then
\begin{equation}\label{cv11}
\sup_{\eta \le x} \int_{-\infty}^{\infty} \left|\frac{\Phi(x+iy)}{x+iy}\right| e^{-k\pi |y|}\ \dy < \infty,
\end{equation}
and
\begin{equation}\label{cv12}
\lim_{x\rightarrow \infty} \int_{-\infty}^{\infty} \left|\frac{\Phi(x+iy)}{x+iy}\right| e^{-k\pi |y|}\ \dy = 0.
\end{equation}

\begin{lemma}\label{lem2.1}  Assume that the analytic function $\Phi:\rr\rightarrow \C$ satisfies the 
conditions {\rm (\ref{cv10})}, {\rm (\ref{cv11})} and {\rm (\ref{cv12})}, and let $0 < \delta$.  Then there 
exists a positive number $c(\delta,\Phi)$, depending only on $\delta$ and $\Phi$, such that the inequality
\begin{equation*}
|\Phi(z)| \le c(\delta,\Phi)\, |z|\, e^{k\pi |y|}
\end{equation*}
holds for all $z = x+iy$ in the closed half plane $\{z\in\C: \delta\le \Re(z)\}$.
\end{lemma}

\begin{proof}
This is \cite[Lemma 2.1]{CV2} for the case $k=2$ and \cite[Lemma 3.1]{CV3} for the case $k=1$. 
\end{proof}

We now let $\beta$ be any positive real number such that $\beta \notin \Z$. From \eqref{cv11} we know that if $z$ is a complex number such that $\Re(z) \neq \beta$, the function 
\begin{equation*}
w \mapsto \left( \frac{\sin \pi z}{\sin \pi w}\right)^k \frac{\Phi (w)}{(z-w)}
\end{equation*}
is integrable along the vertical line $\Re(w) = \beta$. We then define a 
complex valued function $z\mapsto I_k(\beta, \Phi; z)$ on each connected component of the open set $\{z\in \C: \Re(z) \not= \beta\}$ by 
\begin{equation}\label{cv22}
I_k(\beta, \Phi; z) = \frac{1}{2\pi i}\int_{\beta-i\infty}^{\beta+i\infty}\left( \frac{\sin \pi z}{\sin \pi w}\right)^k \frac{\Phi (w)}{(z-w)}\, \dw\,,
\end{equation}
and from Morera's theorem we see that $z \mapsto I_k(\beta, \Phi; z)$ is analytic in each of these components.  

\begin{lemma}\label{lem2.2}
Assume that the analytic function $\Phi:\rr \rightarrow \C$ satisfies the conditions
{\rm (\ref{cv10})}, {\rm (\ref{cv11})} and {\rm (\ref{cv12})}.  Let $\beta$ be a positive real number such that $\beta \notin \Z$,  and $z = x + iy$ be a complex number such that $\Re(z) \neq \beta$. Writing
\begin{equation*}
B(\beta, \Phi) = \frac{2^{k-1}}{\pi} 
	\int_{-\infty}^{+\infty} \left|\frac{\Phi(\beta + iv)}{\beta + iv}\right|e^{-k\pi|v|}\ \dv\,,
\end{equation*}
we have
\begin{equation*}
|I_k(\beta, \Phi; z)| \le B(\beta, \Phi)\, |\csc \pi\beta|^k \,
	\left(1 + \frac{|z|}{\bigl|x-\beta\bigr|}\right)e^{k\pi |y|}.
\end{equation*}
\end{lemma}

\begin{proof}
On the vertical line $\Re(w) = \beta$ we have
\begin{equation}\label{cv30}
\left| \frac{w}{z-w}\right| = \left| -1 + \frac{z}{z-w}\right|\leq 1 + \frac{|z|}{|x - \beta|}.
\end{equation}
Note also that 
\begin{equation}\label{cv31}
|\sin \pi (\beta + iv)|^{-k} \le 2^k \, e^{-k\pi |v|} | \csc  \pi\beta|^k.
\end{equation}
We now use \eqref{cv30} and \eqref{cv31} to bound the right-hand side of \eqref{cv22} and obtain the desired result.
\end{proof}

For each positive number $\xi$ we define two rational functions $z\mapsto \A_1(\xi, \Phi; z)$ and $z\mapsto \A_2(\xi, \Phi; z)$ on $\C$ by
\begin{equation*}
\A_1(\xi, \Phi; z) = \frac{\Phi(\xi)}{(z- \xi)}\,,
\end{equation*}
and
\begin{equation*}
\A_2(\xi, \Phi; z) = \frac{\Phi(\xi)}{(z- \xi)^2} +  \frac{\Phi'(\xi)}{(z- \xi)}.
\end{equation*}

\begin{proposition}\label{prop2.3}  
Assume that the analytic function $\Phi:\rr \rightarrow \C$ satisfies the conditions
{\rm (\ref{cv10})}, {\rm (\ref{cv11})} and {\rm (\ref{cv12})}.  Then, for $k=1$ or $2$, the sequence of entire functions
\begin{equation}\label{cv33}
\Bigl(\frac{\sin \pi z}{\pi}\Bigr)^k\ \sum_{n=1}^N   (-1)^{nk}\, \A_k(n, \Phi; z),\ \text{where}\ N = 1, 2, 3, \dots ,
\end{equation}
converges uniformly on compact subsets of $\C$ as $N\rightarrow \infty$, and therefore
\begin{equation}\label{cv34}
\F_k(\Phi; z) = \lim_{N\rightarrow \infty}\Bigl(\frac{\sin \pi z}{\pi}\Bigr)^k \ \sum_{n=1}^N (-1)^{nk}\, \A_k(n, \Phi; z)
\end{equation}
defines an entire function.  
\end{proposition}

\begin{proof} Let $z$ is a complex number in the open right half plane $\rr$ such that $z \notin \Z$.  Then
\begin{equation}\label{cv37}
w \mapsto \left( \frac{\sin \pi z}{\sin \pi w}\right)^k \frac{\Phi (w)}{(z-w)}
\end{equation}
defines a meromorphic function of $w$ on $\rr$.  Note that (\ref{cv37}) has a
simple pole at $w = z$ with residue $-\Phi(z)$.  Also, for each positive integer $n$, (\ref{cv37}) has a 
pole of order at most $k$ at $w = n$ with residue
\begin{equation*}
\Bigl(\frac{\sin \pi z}{\pi}\Bigr)^k (-1)^{nk} \A_k(n, \Phi; z). 
\end{equation*} 
Plainly (\ref{cv37}) has no other poles in $\rr$.  Let $0 < \beta < 1$ and let $N$ be a positive integer and $T$ be a positive real parameter.  Write $\Gamma\big(\beta, N+\hh, T\big)$ for the simply connected, positively oriented rectangular path connecting the points $\beta - iT$, $\big(N+\hh\big) - iT$, $\big(N+\hh\big) + iT$, $\beta + iT$ and $\beta - iT$. If $z$ satisfies $\beta < \Re(z) < N +\hh$ and $|\Im(z)| < T$, and $z$ is not an integer, from the residue theorem we obtain the identity
\begin{align}\label{cv38}
\begin{split}
\Bigl(\frac{\sin \pi z}{\pi}\Bigr)^k &\sum_{n=1}^N (-1)^{nk} \,\A_k(n, \Phi; z) - \Phi(z) \\
	&=\frac{1}{2\pi i}\int_{\Gamma(\beta, N+\frac12, T)} \Bigl(\frac{\sin \pi z}{\sin \pi w}
		\Bigr)^k\frac{\Phi (w)}{(z-w)}\ \dw.
\end{split}
\end{align}
Now let $T\rightarrow \infty$ on the right-hand side of (\ref{cv38}). From hypotheses (\ref{cv10})
and (\ref{cv11}) we obtain
\begin{equation}\label{cv39}
\Bigl(\frac{\sin \pi z}{\pi}\Bigr)^k \sum_{n=1}^N (-1)^{nk} \,\A_k(n, \Phi; z) - \Phi(z) = I_k\big(N+\hh, \Phi; z\big) - I_k(\beta, \Phi; z).
\end{equation}
Initially (\ref{cv39}) holds for $\beta < \Re(z) < N+\hh$ and $z \notin\Z$.  However, since we have both sides of (\ref{cv39}) analytic in the strip $\{z\in \C: \beta < \Re(z) < N+\hh\}$, the condition $z \notin \Z$ can be dropped. Now let $M < N$ be positive integers and use (\ref{cv39}) to get
\begin{equation}\label{cv40}
\Bigl(\frac{\sin \pi z}{\pi}\Bigr)^k \sum_{n=M+1}^N (-1)^{nk} \,\A_k(n, \Phi; z)= I_k\big(N+\hh, \Phi; z\big) - I_k\big(M+\hh, \Phi; z\big)
\end{equation}
in the infinite strip $\{z\in \C: \beta < \Re(z) < M+\hh\}$.  In fact, both sides of (\ref{cv40})
are analytic in $\{z\in \C: \Re(z) < M + \hh\}$, and the identity (\ref{cv40}) must hold in
this larger domain by analytic continuation.  Let $\K\subseteq \C$ be a compact set and assume that $L$ is an integer so large that $\K \subseteq \{z\in \C: |z| < L/2\}$.  From (\ref{cv12}), Lemma \ref{lem2.2} and 
(\ref{cv40}), we see that the sequence of entire functions (\ref{cv33}), where $L \le N$,
is uniformly Cauchy on $\K$.  This verifies the assertion of the lemma and shows that (\ref{cv34}) 
defines an entire function.  
\end{proof}

\begin{lemma}\label{lem2.4}  Assume that the analytic function $\Phi:\rr \rightarrow \C$ satisfies the conditions {\rm (\ref{cv10})}, {\rm (\ref{cv11})} and {\rm (\ref{cv12})}.  For $k=1$ or $2$, let the entire function $\F_k(\Phi; z)$ be defined by {\rm (\ref{cv34})} and let $0 < \beta < 1$. For $\beta < \Re(z)$ we have
\begin{equation}\label{cv41}
\Phi(z) - \F_k(\Phi; z) =  I_k(\beta, \Phi; z),
\end{equation} 
and for $\Re(z) < \beta$ we have
\begin{equation}\label{cv42}
-\F_k(\Phi; z) =  I_k(\beta, \Phi; z).
\end{equation}
\end{lemma}

\begin{proof}  For $\beta < \Re(z)$ we let $N\rightarrow \infty$ on both sides of
(\ref{cv39}), and use (\ref{cv12}) and Lemma \ref{lem2.2} to obtain 
\begin{equation*}\label{cv45}
\Phi(z) - \F_k(\Phi; z) = I_k(\beta, \Phi; z).
\end{equation*}
For $\Re(z) < \beta$, the residue theorem would give us
\begin{equation}\label{cv46}
\Bigl(\frac{\sin \pi z}{\pi}\Bigr)^k \sum_{n=1}^N (-1)^{nk} \A_k(n, \Phi; z) =\frac{1}{2\pi i}\int_{\Gamma(\beta, N+\frac12, T)} \Bigl(\frac{\sin \pi z}{\sin \pi w}
		\Bigr)^k\frac{\Phi (w)}{(z-w)}\, \dw.
\end{equation}
Arguing as before, we let $T\rightarrow \infty$ and then $N\to \infty$ to get 
\begin{equation*}
-\F_k(\Phi; z) = I_k(\beta, \Phi; z).
\end{equation*}
\end{proof}

\begin{lemma}\label{lem2.6}
Assume that the analytic function $\Phi:\rr \rightarrow \C$ satisfies the conditions
{\rm (\ref{cv10})}, {\rm (\ref{cv11})} and {\rm (\ref{cv12})}.  For $k=1$ or $2$, let the entire function $\F_k(\Phi; z)$ be defined by {\rm (\ref{cv34})}. Then there exists a positive number 
$c_k(\Phi)$ such that the inequality
\begin{equation}\label{cv51}
|\F_k(\Phi; z)|\le c_k(\Phi)(1 +|z|)e^{k\pi |y|}
\end{equation}
holds for all complex numbers $z=x+iy$.  In particular, $\F_k(\Phi; z)$ is an entire function
of exponential type at most $k\pi$.
\end{lemma}

\begin{proof}
In the closed half plane $\{z\in \C: \frac12 \le \Re(z)\}$ the identity (\ref{cv41}) implies that
\begin{equation*}
|\F_k(\Phi; z)| \le |\Phi(z)| + |I_k(\tfrac{1}{4}, \Phi; z)|.
\end{equation*}
Then an estimate of the form (\ref{cv51}) in this half plane follows from Lemma \ref{lem2.1} and Lemma \ref{lem2.2}.
In the closed half plane $\{z\in \C: \Re(z) \le \frac12\}$ we have
\begin{equation*}
|\F_k(\Phi; z)| = |I_k(\tfrac{3}{4}, \Phi; z)|
\end{equation*}
from the identity (\ref{cv42}), and an estimate of the form (\ref{cv51}) follows directly from Lemma \ref{lem2.2}. 
\end{proof}

\section{Proofs of Theorems \ref{ba-nu-thm} and \ref{maj-nu-thm}}

\subsection{Preliminaries} Let $\lambda >0$ and $a >0$ be two real parameters. In this proof, instead of writing $T_{a\lambda, e^{-\lambda}}(x)$ repeatedly,  we simplify the notation by defining for $z \in \C$
\begin{equation*}
\T_{\lambda}(a;z) =  \left\{
\begin{array}{cc}
e^{-a \lambda  z} - e^{-\lambda} & \textrm{if} \ \ \Re(z)>0,\\
\frac{1}{2}\big(1 - e^{-\lambda} \big) & \textrm{if} \ \ \Re(z)=0,\\
0 & \textrm{if} \ \ \Re(z)<0.
\end{array}
\right.
\end{equation*}
Note that $\T_{\lambda}(a;x) = T_{a\lambda, e^{-\lambda}}(x)$ when $x \in \R$. Throughout this proof we write $\T_{\lambda}'(a;z)$ for $\frac{\partial}{\partial z} \T_{\lambda}(a;z) $. We have seen in Theorem \ref{tr-ba-thm} that the entire function 
\begin{align*}
\K_{\lambda}(a;z)& =\frac{\sin\pi z}{\pi} \sum_{n=1}^\infty (-1)^{n} \frac{\T_{\lambda}(a;n)}{(z-n)}+ \frac{\sin\pi z}{\pi z} \left\{ -\frac{e^{-\lambda}}{2} + \sum_{n=1}^\infty (-1)^{n+1}  e^{-a \lambda n} \right\}\\
& := \G_{\lambda}(a;z) + \K^{\star}_{\lambda}(a;z)
\end{align*}
is the best approximation of exponential type $\pi$ of $\T_{\lambda}(a;x)$ for any $a \leq 1$. Also, in Theorem \ref{tr-majorant-thm} we have shown that the entire function 
\begin{align*}
\L_{\lambda}(a;z) & = \frac{\sin^2\pi z}{\pi^2}\sum_{n=1}^\infty \left(\frac{\T_{\lambda}(a;n)}{(z-n)^2} + \frac{\T_{\lambda}'(a;n)}{(z-n)}\right)  + \frac{\sin^2\pi z}{\pi^2 z}\left\{ -e^{-\lambda } -\sum_{n=1}^\infty  \T_{\lambda}'(a;n) \right\}\\
& := \H_{\lambda}(a;z) + \L^{\star}_{\lambda}(a;z)
\end{align*}
is the extremal minorant of exponential type $2\pi$ of $\T_{\lambda}(a;x)$ for any $a \leq 1$, and that
\begin{align*}
\begin{split}
\M_{\lambda}(a;z) & = \frac{\sin^2\pi z}{\pi^2} \left\{\sum_{n=1}^\infty \left(\frac{\T_{\lambda}(a;n)}{(z-n)^2} + \frac{\T_{\lambda}'(a;n)}{(z-n)}\right) \right\} \\
& \qquad \ \ \ \   +  \left[\frac{\sin^2\pi z}{\pi^2 z}\left\{ -e^{-\lambda } -\sum_{n=1}^\infty  \T_{\lambda}'(a;n) \right\} + \frac{\sin^2\pi z}{\pi^2z^2} \big(1-e^{-\lambda}\big)\right]\\
& =: \H_{\lambda}(a;z) + \M^{\star}_{\lambda}(a;z)
\end{split}
\end{align*}
is the extremal majorant of exponential type $2\pi$ of $\T_{\lambda}(a;x)$ for any $a > 0$.

\smallskip

It will be useful to analyze the growth of the functions $\G_{\lambda}(a;z)$ and $\H_{\lambda}(a;z)$ when we restrict $\lambda$ to a compact interval. Letting $N>1$ and $\lambda \in [\frac{1}{N},N]$, the next lemma states that we can find bounds that depend only on $a$ and $N$.

\begin{lemma}\label{lem_unif_bound}
Given $N>1$ and $a >0$, there exist constants $c_1(a, N)$ and $c_2(a, N)$ such that 
\begin{equation}\label{prelim1}
\big|\G_{\lambda}(a;z)\big| \leq c_1(a, N) \,(1 + |z|) \,e^{\pi |y|},
\end{equation}
and
\begin{equation}\label{prelim2}
\big|\H_{\lambda}(a;z)\big| \leq c_2(a, N) \,(1 + |z|) \,e^{2 \pi |y|},
\end{equation}
for all complex numbers $z = x+iy$ and $\lambda \in [\tfrac{1}{N},N]$.
\end{lemma}

\begin{proof}
Observe first that 
\begin{align}\label{prelim3}
\begin{split}
\big|\G_{\lambda}(a;z)\big| & \leq \left| \frac{\sin \pi z}{\pi} \sum_{n=1}^{\infty} (-1)^n\frac{ e^{-a \lambda n}}{(z-n)} \right| +  e^{-\lambda}  \left| \frac{\sin \pi z}{\pi} \sum_{n=1}^{\infty} \frac{(-1)^n}{(z-n)} \right| \\
& \leq   \frac{|\sin \pi z|}{\pi} \sum_{n=1}^{\infty} \frac{ e^{-\frac{an}{N}}}{|z-n|} + e^{-\frac{1}{N}} \left| \frac{\sin \pi z}{\pi} \sum_{n=1}^{\infty} \frac{(-1)^n}{(z-n)} \right|\,,
\end{split}
\end{align}
and from \eqref{prelim3} we can get to \eqref{prelim1}. In a similar way we have
\begin{align}\label{prelim4}
\begin{split}
\big|\H_{\lambda}(a;z)\big| & \leq \frac{\big|\sin^2 \pi z\big|}{\pi} \sum_{n=1}^{\infty} \left\{ \frac{ \big|e^{-a\lambda n} - e^{-\lambda}\big|}{|z-n|^2} +  \frac{ a \lambda e^{-a \lambda n}}{|z-n|} \right\}\\
&\ \ \ \ \ \ \ \ \ \  \leq  \frac{\big|\sin^2 \pi z\big|}{\pi} \sum_{n=1}^{\infty} \left\{ \frac{ e^{-\frac{an}{N}} + e^{-\frac{1}{N}}}{|z-n|^2} +  \frac{ aN e^{-\frac{an}{N}}}{|z-n|} \right\}\,,
\end{split}
\end{align}
and from \eqref{prelim4} we get to \eqref{prelim2}. 
\end{proof}

Let $\nu$ be a nonnegative measure defined on the Borel subsets of $(0,\infty)$ that satisfies \eqref{min-nu-growth} or \eqref{maj-nu-growth}. For $z \in \C$ we write
\begin{equation}\label{pt0}
\TT_{\nu}(a; z) = \int_0^{\infty} \T_{\lambda}(a;z) \,\dnu.
\end{equation}
Observe that $z \mapsto \TT_{\nu}(a; z)$ is analytic in the open right half plane $\mc{R} = \{z \in \C; 0 < \Re(z)\}$ (from Morera's theorem) and it might take the value $+\infty$ at $\Re(z) =0$.

\begin{lemma}\label{CL_lem7}  For $k=1$ or $2$, and $a>0$, the analytic function $\TT_{\nu}(a; \,\cdot\,): \mc{R} \to \C$ defined by {\rm (\ref{pt0})} satisfies each of the three conditions {\rm (\ref{cv10})}, {\rm (\ref{cv11})} and {\rm (\ref{cv12})}.
\end{lemma}
\begin{proof}
This is \cite[Lemma 4.1]{CV2} for the case $k=2$ and \cite[Lemma 5.1]{CV3} for the case $k=1$. 
\end{proof}

\subsection{Proof of Theorem \ref{ba-nu-thm}} Let $\nu$ be a measure satisfying \eqref{min-nu-growth} and $\delta \geq 1$. Recall that finding the optimal approximation of type $\pi \delta$ for $\TT_{\nu}(x)$ is equivalent, by scaling, to the problem of finding the optimal approximation of type $\pi$ for $\TT_{\nu}\big(\delta^{-1}; x\big)$, and we will solve the latter. We define
\begin{align*}
\KK_{\nu}\big(\delta^{-1};z\big)& =\frac{\sin\pi z}{\pi} \left\{ \sum_{n=1}^\infty (-1)^{n} \frac{\TT_{\nu}\big(\delta^{-1};n\big)}{(z-n)}\right\} \\
&  \ \ \ \ \ \ \ \ \ \ \ \ \ \ \  \ \ \  + \frac{\sin\pi z}{\pi z} \left\{ \int_0^{\infty}\left(-\frac{e^{-\lambda}}{2} + \sum_{n=1}^\infty (-1)^{n+1}  e^{-\frac{\lambda}{\delta} n}\right)\dnu \right\}\\
& := \GG_{\nu}\big(\delta^{-1};z\big) + \KK^{\star}_{\nu}\big(\delta^{-1}; z\big),
\end{align*}
and we aim to show that this is the unique best approximation of exponential type $\pi$ of $\TT_{\nu}\big(\delta^{-1}; x\big)$.

\begin{lemma} Let $\delta\ge 1$. The function $z\mapsto \KK_{\nu}\big(\delta^{-1};z\big)$ is an entire function of exponential type $\pi$ which satisfies
\begin{align}\label{CL_pf3-lem}
\begin{split}
 \TT_{\nu}\big(\delta^{-1};x\big)  &-\KK_{\nu}\big(\delta^{-1};x\big) = \int_0^\infty  \big\{\T_{\lambda} \big(\delta^{-1};x\big)  - \K_{\lambda}\big(\delta^{-1};x\big)\big\}\, \dnu
\end{split}
\end{align}
for every real $x$. 
\end{lemma}

\begin{proof} 
Note first that 
\begin{equation}\label{CL_ineq_pos}
-\frac{e^{-\lambda}}{2} + \sum_{n=1}^\infty (-1)^{n+1}  e^{-\frac{\lambda}{\delta} n} = \frac{2 e^{-\frac{\lambda}{\delta}} - e^{-\lambda} - e^{-\lambda - \frac{\lambda}{\delta}}}{2 \big( 1 + e^{-\frac{\lambda}{\delta}}\big)} \geq 0
\end{equation}
for any $\lambda >0 $ and $\delta \geq 1$, and this is integrable with respect to the measure $\nu$. Therefore $ \KK^{\star}_{\nu}\big(\delta^{-1}; z\big)$ is an entire function of exponential type $\pi$. From the general framework provided by  Proposition \ref{prop2.3} and Lemma \ref{lem2.6}, together with Lemma \ref{CL_lem7},  we also know that $\GG_{\nu}\big(\delta^{-1};z\big)$ is an entire function of exponential type $\pi$, and therefore so is $\KK_{\nu}\big(\delta^{-1};z\big)$.

For $N>1$ we shall consider truncations $\nu_N$ of the measure $\nu$, by restricting it to the interval $[\tfrac{1}{N},N]$. In particular, from \eqref{min-nu-growth} we know that $\nu_N$ is a finite measure. Let us write
\begin{equation}\label{def_T_N}
\TT_{N}\big(\delta^{-1};z\big) = \TT_{\nu_N}\big(\delta^{-1};z\big) \!=\!  \int_{0}^{\infty} \T_{\lambda}\big(\delta^{-1};z\big) \, \dnuN =  \int_{\frac{1}{N}}^{N} \T_{\lambda}\big(\delta^{-1};z\big) \, \dnu,
\end{equation}
and define
\begin{equation*}
\GG_{N}\big(\delta^{-1};z\big) = \int_{0}^{\infty} \G_{\lambda}\big(\delta^{-1};z\big) \, \dnuN =  \int_{\frac{1}{N}}^{N} \G_{\lambda}\big(\delta^{-1};z\big) \, \dnu,
\end{equation*}
\begin{equation*}
\KK^{\star}_{N}\big(\delta^{-1};z\big) = \int_{0}^{\infty}\K^{\star}_{\lambda}\big(\delta^{-1};z\big) \, \dnuN =  \int_{\frac{1}{N}}^{N}\K^{\star}_{\lambda}\big(\delta^{-1};z\big) \, \dnu,
\end{equation*}
and
\begin{equation*}
\KK_{N}\big(\delta^{-1};z\big) = \GG_{N}\big(\delta^{-1};z\big)  + \KK^{\star}_{N}\big(\delta^{-1};z\big).
\end{equation*}
From Lemma \ref{lem_unif_bound} we find that $\lambda \mapsto \big|\G_{\lambda}\big(\delta^{-1};z\big)\big|$ is integrable with respect to the finite measure $\nu_N$ and thus $\GG_{N}\big(\delta^{-1};z\big)$ is an entire function of exponential type $\pi$. Similarly, it follows from \eqref{CL_ineq_pos} that $\KK^{\star}_{N}\big(\delta^{-1};z\big)$ is an entire function of exponential type $\pi$, hence $\KK_{N}\big(\delta^{-1};z\big)$ is also an entire function of exponential type $\pi$. In fact, it can be shown that $\KK_{N}\big(\delta^{-1};z\big)$ is the solution of the best approximation problem of type $\pi$ for $\TT_{N}\big(\delta^{-1};x\big)$ but we shall not use this particular fact here.

\smallskip

Let $0< \beta <1$. For $\beta < \Re(z)$, using Lemma \ref{lem2.4} and Fubini's theorem we have
\begin{align}\label{CL_pf0}
\begin{split}
\TT_{N}&\big(\delta^{-1};z\big)  - \KK_{N}\big(\delta^{-1};z\big)\\
&  = \big(\TT_{N}\big(\delta^{-1};z\big) - \GG_{N}\big(\delta^{-1};z\big)\big)  - \KK^{\star}_{N}\big(\delta^{-1};z\big) \\
& =    \int_{\frac{1}{N}}^{N}\big\{ \T_{\lambda}\big(\delta^{-1};z\big) -\G_{\lambda}\big(\delta^{-1};z\big)\big\}   \, \dnu - \KK^{\star}_{N}\big(\delta^{-1};z\big) \\
& = \int_{\frac{1}{N}}^{N}\left\{ \frac{1}{2\pi i}\int_{\beta-i\infty}^{\beta+i\infty}\left( \frac{\sin \pi z}{\sin \pi w}\right) \frac{\T_{\lambda}\big(\delta^{-1};w\big)}{(z-w)}\, \dw\right\}\dnu  - \KK^{\star}_{N}\big(\delta^{-1};z\big)\\
& = \frac{1}{2\pi i} \int_{\beta-i\infty}^{\beta+i\infty}\left( \frac{\sin \pi z}{\sin \pi w}\right)\frac{\TT_{N}\big(\delta^{-1};w\big)}{(z-w)}\, \dw - \KK^{\star}_{N}\big(\delta^{-1};z\big).
\end{split}
\end{align}
Now observe that 
\begin{align}\label{CL_pf1}
\begin{split}
\big|\TT_{N}\big(\delta^{-1};w\big)\big| & \leq \int_0^{\infty} \big| \T_{\lambda}\big(\delta^{-1};w\big)\big| \,\dnu \\
&=  \int_0^{\infty} \left| \int_{\delta}^{w} \T'_{\lambda}\big(\delta^{-1};z\big)\,\dz\right| \dnu \\
& =  \int_0^{\infty} \left| \int_{\delta}^{w} -\tfrac{\lambda}{\delta}\, e^{-\frac{\lambda}{\delta}z}\,\dz\right| \dnu \\
&\leq (|w| + \delta) \int_0^{\infty} \tfrac{\lambda}{\delta}\, e^{-\frac{\lambda}{\delta}\beta}\, \dnu\\
& = (|w| + \delta) \, \big|\TT'_{\nu}\big(\delta^{-1};\beta\big)\big|.
\end{split}
\end{align}
From \eqref{CL_pf1} we can apply dominated convergence in \eqref{CL_pf0}. Using Lemma \ref{lem2.4} again we have
\begin{align}\label{CL_pf2}
\begin{split}
\lim_{N\to \infty} & \big\{\TT_{N} \big(\delta^{-1};z\big)  - \KK_{N}\big(\delta^{-1};z\big)\big\}  \\
&  =  \frac{1}{2\pi i} \int_{\beta-i\infty}^{\beta+i\infty}\!\!\left( \frac{\sin \pi z}{\sin \pi w}\right)\frac{\TT_{\nu}\big(\delta^{-1};w\big)}{(z-w)} \dw - \KK^{\star}_{\nu}\big(\delta^{-1};z\big)\\
&  = \big(\TT_{\nu}\big(\delta^{-1};z\big) - \GG_{\nu}\big(\delta^{-1};z\big)\big)  - \KK^{\star}_{\nu}\big(\delta^{-1};z\big) \\
&= \TT_{\nu}\big(\delta^{-1};z\big)  -\KK_{\nu}\big(\delta^{-1};z\big).
\end{split}
\end{align}
In particular, for any positive $x$, using \eqref{th1-eq1} we can apply the monotone convergence theorem to get 
\begin{align}\label{CL_pf3}
\begin{split}
 \TT_{\nu}\big(\delta^{-1};x\big)  &-\KK_{\nu}\big(\delta^{-1};x\big) \\
&= \lim_{N\to \infty} \big\{\TT_{N} \big(\delta^{-1};x\big)  - \KK_{N}\big(\delta^{-1};x\big)\big\}\\
 & =    \lim_{N\to \infty}  \int_{\frac{1}{N}}^{N} \big\{\T_{\lambda} \big(\delta^{-1};x\big)  - \K_{\lambda}\big(\delta^{-1};x\big)\big\}\, \dnu \\
 &= \int_0^\infty  \big\{\T_{\lambda} \big(\delta^{-1};x\big)  - \K_{\lambda}\big(\delta^{-1};x\big)\big\}\, \dnu.
\end{split}
\end{align}
For $\Re(z) <0$, we apply instead equation \eqref{cv42} in Lemma \ref{lem2.4} to obtain \eqref{CL_pf0} and \eqref{CL_pf2} in the exact same manner. Therefore \eqref{CL_pf3} also holds for any negative $x$.
\end{proof}

We now have all the ingredients to prove Theorem \ref{ba-nu-thm}. From \eqref{CL_pf3-lem} we conclude that 
\begin{equation*}
\sin \pi x \,\big\{ \TT_{\nu}\big(\delta^{-1};x\big)  -\KK_{\nu}\big(\delta^{-1};x\big)\big\} \geq 0 
\end{equation*}
for all real $x \neq 0$, and also
\begin{align*}
\int_{-\infty}^{\infty} \big| \TT_{\nu}\big(\delta^{-1};x\big)  -\KK_{\nu}& \big(\delta^{-1};x\big)\big|\, \dx  =  \int_{-\infty}^{\infty} \int_0^\infty \big|\T_{\lambda} \big(\delta^{-1};x\big)  - \K_{\lambda}\big(\delta^{-1};x\big)\big|\, \dnu\,  \dx \\
& \ \ \ \  = \int_0^\infty \left( \frac{1-e^{-\delta^{-1}\lambda}}{\delta^{-1}\lambda \big(1+e^{-\delta^{-1}\lambda}\big)} - \frac{e^{-\lambda}}{2}\right)\dnu.
\end{align*}

The argument to prove that this is indeed the minimum possible value, and that the best approximation $z \mapsto \KK_{\nu}\big(\delta^{-1};z\big)$ is unique, is analogous to the proof of Theorem \ref{tr-ba-thm}, and we omit the details.  This concludes the proof of Theorem \ref{ba-nu-thm}.

\subsection{Proof of Theorem  \ref{maj-nu-thm} - Minorant} We will follow the same strategy designed for the proof of Theorem \ref{ba-nu-thm}. We keep working here under the assumptions that $\delta \geq 1$ and that the measure $\nu$ satisfies  \eqref{min-nu-growth}. Recall that we want to show that the function
\begin{align}\label{CL_pf4_2}
\begin{split}
\LL_{\nu}\big(\delta^{-1};z\big) & = \frac{\sin^2\pi z}{\pi^2}\left\{ \sum_{n=1}^\infty \left(\frac{\TT_{\nu}\big(\delta^{-1};n\big)}{(z-n)^2} + \frac{\TT_{\nu}'\big(\delta^{-1};n\big)}{(z-n)}\right) \right\} \\
& \ \ \ \ \ \ \ \ \ \ \ \ \ + \frac{\sin^2\pi z}{\pi^2 z}\left\{ \int_0^{\infty} \left(-e^{-\lambda } -\sum_{n=1}^\infty  \T_{\lambda}'\big(\delta^{-1};n\big)\right) \dnu \right\}\\
& := \HH_{\lambda}\big(\delta^{-1};z\big) + \LL^{\star}_{\lambda}\big(\delta^{-1};z\big)
\end{split}
\end{align}
is the unique extremal minorant of exponential type $2\pi$ of $\TT_{\nu}\big(\delta^{-1}; x\big)$.

\begin{lemma} Let $\delta\ge 1$. The function $z\mapsto \LL_{\nu}\big(\delta^{-1};z\big)$ is an entire function of exponential type $2\pi$ which satisfies
\begin{align}\label{CL_pf3_lem2}
\begin{split}
 \TT_{\nu}\big(\delta^{-1};x\big) & -\LL_{\nu}\big(\delta^{-1};x\big) = \int_0^\infty \big\{\T_{\lambda}\big(\delta^{-1};x\big)  - \L_{\lambda}\big(\delta^{-1};x\big)\big\}\, \dnu.
\end{split}
\end{align}
for every real $x$. 
\end{lemma}

\begin{proof} We observe first that 
\begin{align}\label{CL_ineq_pos_2}
-e^{-\lambda } -\sum_{n=1}^\infty  \T_{\lambda}'\big(\delta^{-1};n\big) = \frac{ -e^{-\lambda} + e^{-\lambda - \frac{\lambda}{\delta}} + \frac{\lambda}{\delta} e^{-\frac{\lambda}{\delta}}}{ 1 - e^{-\frac{\lambda}{\delta}}} \geq 0
\end{align}
for any $\lambda >0 $ and $\delta \geq 1$, and that this is integrable with respect to the measure $\nu$. Therefore $\LL^{\star}_{\lambda}\big(\delta^{-1};z\big)
$ is an entire function of exponential type $2\pi$. From Proposition \ref{prop2.3} and Lemma \ref{lem2.6}, together with Lemma \ref{CL_lem7},  we know that $\HH_{\nu}\big(\delta^{-1};z\big)$ is also an entire function of exponential type $2\pi$, and thus so is $\LL_{\nu}\big(\delta^{-1};z\big)$.

\smallskip

With $\TT_N\big(\delta^{-1};z\big)$ as in \eqref{def_T_N} we now define
\begin{equation*}
\HH_{N}\big(\delta^{-1};z\big) = \int_{0}^{\infty} \H_{\lambda}\big(\delta^{-1};z\big) \, \dnuN =  \int_{\frac{1}{N}}^{N} \H_{\lambda}\big(\delta^{-1};z\big) \, \dnu,
\end{equation*}
\begin{equation*}
\LL^{\star}_{N}\big(\delta^{-1};z\big) = \int_{0}^{\infty} \L^{\star}_{\lambda}\big(\delta^{-1};z\big) 
\, \dnuN =  \int_{\frac{1}{N}}^{N} \L^{\star}_{\lambda}\big(\delta^{-1};z\big) 
\, \dnu,
\end{equation*}
and
\begin{equation*}
\LL_{N}\big(\delta^{-1};z\big) = \HH_{N}\big(\delta^{-1};z\big)  + \LL^{\star}_{N}\big(\delta^{-1};z\big).
\end{equation*}
As before, from Lemma \ref{lem_unif_bound} we find that $\big|\H_{\lambda}\big(\delta^{-1};z\big)\big|$ is integrable with respect to the finite measure $\nu_N$ and thus $\HH_{N}\big(\delta^{-1};z\big)$ is a well defined entire function of exponential type $2\pi$. The analogous statement is true for $\LL^{\star}_{N}\big(\delta^{-1};z\big)$ by \eqref{CL_ineq_pos_2}, and thus $\LL_{N}\big(\delta^{-1};z\big)$ is an entire function of exponential type $2\pi$.  The entire function $\LL_{N}\big(\delta^{-1};z\big)$ turns out to be the unique extremal minorant of type $2\pi$ for $\TT_{N}\big(\delta^{-1};x\big)$, but this is not required for the following, and we do not prove it here.

\smallskip 

Let $0< \beta <1$. For $\beta < \Re(z)$, using Lemma \ref{lem2.4} and Fubini's theorem we have
\begin{align}\label{CL_pf0_2}
\begin{split}
\TT_{N}& \big(\delta^{-1};z\big)  - \LL_{N}\big(\delta^{-1};z\big)\\
&  = \big(\TT_{N}\big(\delta^{-1};z\big) - \HH_{N}\big(\delta^{-1};z\big)\big)  - \LL^{\star}_{N}\big(\delta^{-1};z\big) \\
& =    \int_{\frac{1}{N}}^{N}\big\{ \T_{\lambda}\big(\delta^{-1};z\big) -\H_{\lambda}\big(\delta^{-1};z\big)\big\}   \, \dnu - \LL^{\star}_{N}\big(\delta^{-1};z\big) \\
& = \int_{\frac{1}{N}}^{N}\left\{ \frac{1}{2\pi i}\int_{\beta-i\infty}^{\beta+i\infty}\left( \frac{\sin \pi z}{\sin \pi w}\right)^2 \frac{\T_{\lambda}\big(\delta^{-1};w\big)}{(z-w)}\, \dw\right\}\dnu  - \LL^{\star}_{N}\big(\delta^{-1};z\big)\\
& = \frac{1}{2\pi i} \int_{\beta-i\infty}^{\beta+i\infty}\left( \frac{\sin \pi z}{\sin \pi w}\right)^2\frac{\TT_{N}\big(\delta^{-1};w\big)}{(z-w)}\, \dw - \LL^{\star}_{N}\big(\delta^{-1};z\big).
\end{split}
\end{align}
Using \eqref{CL_pf1} again, we may let $N \to \infty$ in \eqref{CL_pf0_2} using dominated convergence. Another application of Lemma \ref{lem2.4} gives us
\begin{align}\label{CL_pf2_2}
\begin{split}
 \lim_{N\to \infty} & \big\{\TT_{N} \big(\delta^{-1};z\big)  - \LL_{N}\big(\delta^{-1};z\big)\big\} \\
&  = \frac{1}{2\pi i} \int_{\beta-i\infty}^{\beta+i\infty}\left( \frac{\sin \pi z}{\sin \pi w}\right)^2\frac{\TT_{\nu}\big(\delta^{-1};w\big)}{(z-w)} \dw - \LL^{\star}_{\lambda}\big(\delta^{-1};z\big) 
\\
& = \big(\TT_{\nu}\big(\delta^{-1};z\big) - \HH_{\nu}\big(\delta^{-1};z\big)\big)  - \LL^{\star}_{\lambda}\big(\delta^{-1};z\big) \\
& = \TT_{\nu}\big(\delta^{-1};z\big)  -\LL_{\nu}\big(\delta^{-1};z\big).
\end{split}
\end{align}
In particular, for any positive $x$, using \eqref{th2-eq1} we apply the monotone convergence theorem to get 
\begin{align}\label{CL_pf3_2}
\begin{split}
 \TT_{\nu}\big(\delta^{-1};x\big)  -\LL_{\nu}\big(\delta^{-1};x\big)& = \lim_{N\to \infty} \big\{\TT_{N} \big(\delta^{-1};x\big)  - \LL_{N}\big(\delta^{-1};x\big)\big\}\\
 & =    \lim_{N\to \infty}  \int_{\frac{1}{N}}^{N} \big\{\T_{\lambda}\big(\delta^{-1};x\big)  - \L_{\lambda}\big(\delta^{-1};x\big)\big\}\,  \dnu \\
 & = \int_0^\infty \big\{\T_{\lambda}\big(\delta^{-1};x\big)  - \L_{\lambda}\big(\delta^{-1};x\big)\big\}\, \dnu.
\end{split}
\end{align}
The same result holds for $\Re(z) <0$, when we apply instead equation \eqref{cv42} in Lemma \ref{lem2.4} to obtain \eqref{CL_pf0_2} and \eqref{CL_pf2_2}. Thus \eqref{CL_pf3_2} also holds for any negative $x$ and the lemma is shown.
\end{proof}

To finish the proof of the minorant case of Theorem \ref{maj-nu-thm} we note that \eqref{CL_pf3_lem2} implies
\begin{equation}\label{CL_pf5_2}
\TT_{\nu}\big(\delta^{-1};x\big)  -\LL_{\nu}\big(\delta^{-1};x\big) \geq 0
\end{equation}
for any $x \neq 0$, and that
\begin{align}\label{CL_pf6_2}
\begin{split}
\int_{-\infty}^{\infty}  \big\{ \TT_{\nu}&\big(\delta^{-1};x\big)  -\LL_{\nu}\big(\delta^{-1};x\big)\big\}\, \dx \\
&  =  \int_{-\infty}^{\infty}\int_0^\infty  \big\{\T_{\lambda}\big(\delta^{-1};x\big)  - \L_{\lambda}\big(\delta^{-1};x\big)\big\} \dnu\,  \dx \\
&  = \int_0^\infty\left(\frac\delta\lambda -\frac{e^{-\lambda}}{2} -\frac{1}{e^{\delta^{-1}\lambda}-1}\right)\dnu.
\end{split}
\end{align}
Equality occurs in \eqref{CL_pf5_2} if $x \in \Z\backslash\{0\}$. From \eqref{CL_pf4_2} we also have $\LL_{\nu}\big(\delta^{-1};0\big) = 0$. This is enough to conclude that \eqref{CL_pf6_2} is indeed the minimal integral, and it is attained only by the entire function $z \mapsto \LL_{\nu}\big(\delta^{-1};z\big)$, as in the proof of Theorem \ref{tr-majorant-thm}.

\subsection{Proof of Theorem \ref{maj-nu-thm} - Majorant} In this case we work with any $\delta >0$. This is justified since we have the shown the result for the base function  $\T_{\lambda}\big(\delta^{-1};x\big)$. We also require the more restrictive condition \eqref{maj-nu-growth} on the measure $\nu$. From \eqref{defMM} our candidate is
\begin{equation*}
\MM_\nu\big(\delta^{-1};z\big) = \LL_\nu\big(\delta^{-1}; z\big) + \left(\frac{\sin\pi z}{\pi z}\right)^2 \int_0^\infty \big(1-e^{-\lambda}\big) \,\dnu.
\end{equation*}
From the minorant case we already know that $\LL_\nu\big(\delta^{-1}; z\big)$ is entire of exponential type $2\pi$. From condition \eqref{maj-nu-growth} we have that $\lambda \mapsto (1 - e^{-\lambda})$ is $\nu$-integrable and thus  $\MM_\nu\big(\delta^{-1}; z\big)$ itself is entire of exponential type $2\pi$. To prove that 
\begin{equation*}
\MM_{\nu}\big(\delta^{-1};x\big) - \TT_{\nu}\big(\delta^{-1};x\big)  \geq 0,
\end{equation*}
we proceed in the same way as in \eqref{CL_pf0_2},  \eqref{CL_pf2_2} and  \eqref{CL_pf3_2}, by taking $0 < \beta < \min\{1,\delta\}$. Also, analogously to \eqref{CL_pf6_2} we find
\begin{align*}
\begin{split}
\int_{-\infty}^{\infty}  \big\{ \MM_{\nu}&\big(\delta^{-1};x\big)  -\TT_{\nu}\big(\delta^{-1};x\big)\big\}\, \dx \\
&  =  \int_{-\infty}^{\infty}\int_0^\infty  \big\{\M_{\lambda}\big(\delta^{-1};x\big)  - \T_{\lambda}\big(\delta^{-1};x\big)\big\} \dnu\,  \dx \\
&  = \int_0^\infty\left(\frac{1}{1-e^{-\delta^{-1}\lambda}} -\frac{\delta}{\lambda} -\frac{e^{-\lambda}}{2}\right)\,\dnu.
\end{split}
\end{align*}
The optimality and uniqueness are similar to the proof of Theorem  \ref{tr-majorant-thm}.  This concludes the proof of Theorem \ref{maj-nu-thm}.


\begin{thebibliography}{99}

\bibitem{BMV}
J.~T.~Barton, H.~L.~Montgomery, and J.~D.~Vaaler,
\newblock Note on a diophantine inequality in several variables,
\newblock Proc. Amer. Math. Soc. 129 (2001), 337--345.

\bibitem{CC}
E. Carneiro and V. Chandee,
\newblock Bounding $\zeta(s)$ in the critical strip, 
\newblock J. Number Theory 131 (2011), 363--384. 

\bibitem{CCM}
E. Carneiro, V. Chandee and M. Milinovich,
\newblock Bounding $S(t)$ and $S_1(t)$ on the Riemann hypothesis,
\newblock Math. Ann. (to appear) DOI 10.1007/s00208-012-0876-z.

\bibitem{CL}
E. Carneiro and F. Littmann,
\newblock Bandlimited approximations to the truncated Gaussian and applications,
\newblock Constr. Approx. (to appear) DOI 10.1007/s00365-012-9177-8.

\bibitem{CLV}
E.\ Carneiro, F.\ Littmann, and J. D.\ Vaaler,
\newblock Gaussian subordination for the Beurling-Selberg extremal problem,
\newblock Trans.\ Amer.\ Math.\ Soc. (to appear) DOI: http://dx.doi.org/10.1090/S0002-9947-2013-05716-9.

\bibitem{CV2}
E.~Carneiro and J.~D.~Vaaler,
\newblock Some extremal functions in Fourier analysis, II,
\newblock Trans. Amer. Math. Soc. 362 (2010), 5803-5843.

\bibitem{CV3}
E.~Carneiro and J.~D.~Vaaler,
\newblock Some extremal functions in Fourier analysis, III,
\newblock Constr. Approx. 31, No. 2 (2010), 259--288.

\bibitem{CS}
V. Chandee and K. Soundararajan,
\newblock Bounding $|\zeta(\tfrac{1}{2} + it)|$ on the Riemann hypothesis,
\newblock Bull. London Math. Soc. 43(2) (2011), 243--250.

\bibitem{Ga}
P. X. Gallagher, 
\newblock Pair correlation of zeros of the zeta function, 
\newblock J. Reine Angew. Math. 362 (1985), 72--86.

\bibitem{Gan} M. I. Ganzburg,
\newblock $L$-approximation to non-periodic functions,
\newblock Journal of concrete and applicable mathematics 8 (2010), no. 2, 208--215.

\bibitem{GG}
D. A. Goldston and S. M. Gonek,
\newblock A note on $S(t)$ and the zeros of the Riemann zeta-function,
\newblock Bull. London Math. Soc. 39 (2007), 482--486.

\bibitem{GV}
S.~W.~Graham and J.~D.~Vaaler,
\newblock A class of extremal functions for the Fourier transform,
\newblock Trans. Amer. Math. Soc. 265 (1981), 283--382.

\bibitem{HZ} U. Haagerup and L. Zsid\'o,
\newblock Resolvent estimate for Hermitian operators and a related minimal extrapolation problem, 
\newblock Acta Sci. Math. (Szeged) 59 (1994), 503--524.

\bibitem{HV} 
J.~Holt and J.~D.~Vaaler,
\newblock The Beurling-Selberg extremal functions for a ball in the Euclidean space,
\newblock Duke Math. Journal 83 (1996), 203--247.

\bibitem{L1} 
F.~Littmann,
\newblock Entire approximations to the truncated powers,
\newblock Constr. Approx. 22 (2005), no. 2, 273--295.

\bibitem{L3}
F.~Littmann,
\newblock Entire majorants via Euler-Maclaurin summation,
\newblock Trans. Amer. Math. Soc. 358 (2006), no. 7, 2821--2836.

\bibitem{L4}
F. Littmann,
\newblock Quadrature and extremal bandlimited functions, 
\newblock SIAM J. Math. Anal. (to appear).

\bibitem{K}
M. G. Krein, 
\newblock On the best approximation of continuous differentiable functions on the whole real axis,
\newblock Dokl. Akad. Nauk SSSR 18 (1938), 615-624 (Russian).

\bibitem{Na}
B.\ Sz.-Nagy,
\newblock \"Uber gewisse Extremalfragen bei transformierten trigonometrischen Entwicklungen II,
\newblock Ber. Math.-Phys. Kl. S\"achs. Akad. Wiss. Leipzig 91, 1939.

\bibitem{PP}
M.~Plancherel and G.~Polya,
\newblock Fonctions enti\'eres et int\'egrales de Fourier multiples (Seconde partie),
\newblock Comment. Math. Helv. 10, (1938), 110--163.

\bibitem{S2}
A.~Selberg,
\newblock Lectures on Sieves, {\em Atle Selberg: Collected Papers}, Vol. II,
\newblock Springer-Verlag, Berlin, 1991, pp. 65--247.

\bibitem{V}
J.~D.~Vaaler,
\newblock Some extremal functions in Fourier analysis,
\newblock Bull. Amer. Math. Soc. 12 (1985), 183--215.

\bibitem{Z} A.\ Zygmund
\newblock Trigonometric series, Vol.\ II,
\newblock Cambridge Univ.\ Press, 1968.


\end{thebibliography}
\end{document}